\documentclass[preprint]{imsart}

\usepackage[utf8]{inputenc} % set input encoding (not needed with XeLaTeX)
\usepackage{amsmath,amsthm,verbatim,amssymb,amsfonts,amscd,graphicx}
\usepackage[colorlinks,citecolor=blue,urlcolor=blue]{hyperref}

\usepackage{mathtools}
\usepackage{geometry} % to change the page dimensions
\usepackage{enumitem}
\usepackage{mathabx} % for \widebar command

\usepackage{graphicx} % support the \includegraphics command and options

\usepackage{booktabs} % for much better looking tables
\usepackage{array} % for better arrays (eg matrices) in maths
\usepackage{verbatim} % adds environment for commenting out blocks of text & for better verbatim
\usepackage{subfig} % make it possible to include more than one captioned figure/table in a single float
\usepackage[utf8]{inputenc}
\usepackage[english]{babel}
\usepackage{sectsty}
\allsectionsfont{\sffamily\mdseries\upshape} % (See the fntguide.pdf for font help)
\renewcommand{\exp}[1]{\operatorname{exp}\left(#1\right)} % Exponential
\providecommand{\argmin}{\mathop\mathrm{arg min}}

\providecommand{\tr}{\mathop\mathrm{tr}}

\providecommand{\sign}{\mathop\mathrm{sign}}

\newcommand\var{\mathrm{Var}}

\newcommand{\eps}{\varepsilon}

\newcommand{\wh}{\widehat}
\newcommand{\wt}{\widetilde}
\newcommand{\h}[1]{h^{(#1)}}

\usepackage[nottoc,notlof,notlot]{tocbibind} % Put the bibliography in the ToC
\usepackage{natbib}
\usepackage{bbm}
\usepackage[titles,subfigure]{tocloft} % Alter the style of the Table of Contents
\usepackage{comment}
\newcommand{\p}{{\rm I}\kern-0.18em{\rm P}}
\newcommand{\1}{{\rm 1}\kern-0.24em{\rm I}}
\newcommand{\E}{{\rm I}\kern-0.18em{\rm E}}
\newcommand{\R}{{\rm I}\kern-0.18em{\rm R}}

\newtheorem{theorem}{Theorem}[section]
\newtheorem{lemma}{Lemma}
\newtheorem{corollary}{Corollary}[section]

\newtheorem{remark}{Remark}
\newtheorem{fact}{Fact}

\newcommand{\rom}[1]{\uppercase\expandafter{\romannumeral #1\relax}}
\newcommand{\be}[1]{\begin{equation*}#1\end{equation*}}
\newcommand{\ben}[1]{\begin{equation}#1\end{equation}}

\def\l{\left}
\def\r{\right}
\newcommand{\m}{\mathcal}
\newcommand{\mb}{\mathbb}

\newcommand{\med}[1]{\mbox{med}\left(#1\right)}

\newcommand{\pr}[1]{\mathbb{P}{\left(#1\right)}}

\newcommand{\pd}{\partial}

\def\al#1{\begin{align*}{#1}\end{align*}}
\def\aln#1{\begin{align}{#1}\end{align}}
\def\ml#1{\begin{multline*}{#1}\end{multline*}}
\def\mln#1{\begin{multline}{#1}\end{multline}}

\allowdisplaybreaks

\begin{document}

\begin{frontmatter}
\title{U-statistics of growing order and sub-Gaussian mean estimators with sharp constants}
%\maketitle
\runauthor{ }
\runtitle{Robust ERM}

%\begin{comment}
\begin{aug}
\author{\fnms{Stanislav} \snm{Minsker} \thanksref{t3}}
%\ead[label=e1]{minsker@usc.edu}
%\and
%\author{\fnms{Xiaohan} \snm{Wei}\thanksref{b,e2}\ead[label=e2,mark]{xiaohanw@usc.edu}}
\address[]{Department of Mathematics, University of Southern California \\
email: \textcolor{blue}{minsker@usc.edu}
%\printead{e1}
}
%\address[b]{Department of Electrical Engineering, University of Southern California, Los Angeles, CA 90089.
%\printead{e2}}

%\thankstext{t2}{Department of Mathematics, University of Southern California}
%\thankstext{t1}{Department of Electrical Engineering, University of Southern California}
\thankstext{t3}{Author acknowledges support by the National Science Foundation grants CIF-1908905 and DMS CAREER-2045068.}
\end{aug}
%\end{comment}

\begin{abstract}
This paper addresses the following question: given a sample of i.i.d. random variables with finite variance, can one construct an estimator of the unknown mean that performs nearly as well as if the data were normally distributed? One of the most popular examples achieving this goal is the median of means estimator. However, it is inefficient in a sense that the constants in the resulting bounds are suboptimal. We show that a permutation-invariant modification of the median of means estimator admits deviation guarantees that are sharp up to $1+o(1)$ factor if the underlying distribution possesses more than $\frac{3+\sqrt{5}}{2}\approx 2.62$ moments and is absolutely continuous with respect to the Lebesgue measure. This result yields potential improvements for a variety of algorithms that rely on the median of means estimator as a building block. 
At the core of our argument is are the new deviation inequalities for the U-statistics of order that is allowed to grow with the sample size, a result that could be of independent interest. 
%Finally, we demonstrate that a hybrid of the median of means and Catoni's estimator is capable of achieving sub-Gaussian deviation guarantees with nearly optimal constants assuming just the existence of the second moment.
%Constructions that achieve sub-Gaussian deviation guarantees are well known by now. 
%We are interested in the estimators that are optimal up to $1+o(1)$ multiplicative 
%devoted to sharp bounds for Several examples of Mean estimation  
\end{abstract}

\begin{keyword}[class=MSC]
\kwd[Primary ]{62G35}
\kwd[; secondary ]{60E15,62E20}
\end{keyword}

\begin{keyword}
\kwd{U-statistics}
\kwd{median-of-means estimator}
\kwd{heavy tails}
%\kwd{asymptotic normality}
%\kwd{consistency}
\end{keyword}
\end{frontmatter}

\mathtoolsset{showonlyrefs}

%################
\section{Introduction.}
\label{sec:intro}
%################

Let $X_1,\ldots,X_N$ be i.i.d. random variables with distribution $P$ having mean $\mu$ and finite variance $\sigma^2$. At the core of this paper is the following question: given $1\leq t \leq t_{\max}(N)$, construct an estimator $\wt \mu_N = \wt \mu_N(X_1,\ldots,X_N)$ such that 
\ben{
\label{eq:sg}
\pr{\l| \wt\mu_N-\mu \r|\geq \sigma\sqrt{\frac t N}} \leq 2 e^{-\frac{t}{L}}
}
for some absolute positive constant $L$. Estimators that satisfy this deviation property are called sub-Gaussian. For example, the sample mean $\bar X_N = \frac{1}{N}\sum_{j=1}^N X_j$ is sub-Gaussian for $t_{\max} \asymp q(N,P)$ where $q(N,P)\to \infty$ as $N\to\infty$ and the constant $L$ equals $2$: this immediately follows from the fact that convergence of the distribution functions is uniform in the central limit theorem. However, $q(N,P)$ can grow arbitrarily slow in general, and it grows as $\log^{1/2}(N)$ if $\mb E|X|^{2+\eps}<\infty$ for some $\eps>0$ in view of the Berry-Esseen theorem \citep[for instance, see the book by][]{Petrov_1975}. At the same time, the so-called median of means (MOM) estimator, originally introduced by \citet[][]{Nemirovski1983Problem-complex00,alon1996space,jerrum1986random} and studied recently in relation to the problem at hand satisfies inequality \eqref{eq:sg} with $t_{\max}$ of order $N$ and $L = 24e$ \citep{lerasle2011robust}, although the latter can be improved. A large body of existing work used the MOM estimator as a core subroutine to relax underlying assumptions for a variety of statistical problems, in particular the methods based on the empirical risk minimization; we refer the reader to an excellent survey paper by \citet{lugosi2019mean} for a detailed overview of the recent advances.

The exact value of constant $L$ in inequality \eqref{eq:sg} is less important in problems where only the minimax rates are of interest, but it becomes crucial in terms of practical value and sample efficiency of the algorithms. The benchmark here is the situation when observations are normally distributed: \citet{catoni2012challenging} showed that no estimator can outperform the sample mean in this situation. The latter satisfies the relation 
\[
\pr{\l| \bar X_N - \mu \r|\geq \sigma\frac{\Phi^{-1}(1-e^{-t/2})}{\sqrt N}} = 2 e^{-\frac{t}{2}}
\]
where $\Phi^{-1}(\cdot)$ denotes the quantile function of the standard normal law. As $\Phi^{-1}(1-e^{-t/2})=(1+o(1))\sqrt{t}$ as $t\to\infty$, the best guarantee of the form \eqref{eq:sg} one can hope for is attained for $L = 2$. It is therefore natural to ask whether there exist sharp sub-Gaussian estimators of the mean, that is, estimators satisfying \eqref{eq:sg} with $L=2(1+o(1))$ where $o(1)$ is a sequence that converges to $0$ as $N\to\infty$, under minimal assumptions on the underlying distribution. This question has previously been posed by \citet{devroye2016sub} as an open problem, and several results appeared since then that give partial answers. We proceed with a brief review of the state of the art. 

%##################################
\subsection{Overview of the existing results.}
%##################################

\citet{catoni2012challenging} presented the first known example of a sharp sub-Gaussian estimator with $t_{\max} = o(N/\kappa)$ for distributions with finite fourth moment and a known upper bound on the kurtosis $\kappa$ (or, alternatively, for distribution with finite but known variance). \citet{devroye2016sub} introduced an alternative estimator that also required finite fourth moment but did not explicitly depend on the value of the kurtosis as an input while satisfying required guarantees for $t_{\max} = o\l( (N/\kappa)^{2/3}\r)$. \citet{minsker2021robust} designed an asymptotically efficient sub-Gaussian estimator $\wt \mu_N$ that satisfies $\sqrt{N}\l( \wt\mu_N - \mu\r)\xrightarrow{d} N(0,\sigma^2)$ assuming only the finite second moment plus a mild, ``small-ball'' type condition. However, the constants in the non-asymptotic version of their bounds were not sharp. 
Finally, \citet{lee2020optimal} constructed an estimator with required properties assuming just the finite second moment, however, their guarantees hold with optimal constants only for $t_{\min}\leq t\leq t_{\max}$ where $t_{\max} = o(N)$ and $t_{\min}\to\infty$ as $N\to\infty$. In particular, this range excludes $t$ in the neighborhood of $0$ which is often the region of most practical interest.

%####################################
\subsection{Summary of the main contributions.}
%####################################

The reasons for the popularity of MOM estimator are plenty: it is simple to define and to compute, it admits strong theoretical guarantees, moreover it is scale-invariant and therefore essentially tuning-free. 
Thus, we believe that any quantifiable improvements to its performance are worth investigating.  

We start by showing that the standard MOM estimator achieves bound \eqref{eq:sg} with $L=\pi(1+o(1))$ where $o(1)\to 0$ as $N\to\infty$; this fact is formally stated in Theorem \ref{th:dev-1}. We then define a permutation-invariant version of MOM, denoted $\wh\mu_N$, and show in Corollary \ref{th:clt-1} that, surprisingly, it is asymptotically optimal in a sense that $\sqrt{N}\l( \wh\mu_N - \mu\r)\xrightarrow{d} N(0,\sigma^2)$ under minimal assumptions; compare this to the the standard MOM estimator that has a limiting variance $\frac{\pi}{2}\sigma^2$. The main result of the paper, Theorem \ref{th:U-mom}, demonstrates that optimality of $\wh\mu_N$ holds in the stronger sense, namely, that inequality \eqref{eq:sg} is valid for a wide range the confidence parameters assuming the distribution of $X_1$ possesses $q$ moments for some possibly unknown $q>\frac{3+\sqrt 5}{2}\approx 2.62$ and that its characteristic function satisfies a mild decay bound. 

Analysis of the estimator $\wh\mu_N$ requires new inequalities for U-statistics of order that grows with the sample size. Detailed discussion and comparison with existing bounds is given in section \ref{sec:growing-order}. In particular, we prove novel bounds for large deviations of the degenerate, higher order terms of the Hoeffding decomposition (Theorem \ref{th:concentration}), and deduce sub-Gaussian deviation guarantees for the non-degenerate U-statistics (Corollary \ref{th:bernstein}) with the ``correct'' sub-Gaussian parameter. These bounds could be of independent interest.
%We believe that these bounds are of independent interest. 
%Finally, in a situation when the distribution $P$ is fixed and the sample size grows, we present an estimator that satisfies sub-Gaussian deviation inequalities with optimal constant $L=2+o(1)$ in the whole range $1\leq t \leq t_{\max} = o(N)$ and requires only the finite variance assumption. Detailed statement is given in Theorem \ref{th:hybrid}.

%#################
\subsection{Notation.}
\label{section:def}
%#################

Unspecified absolute constants will be denoted $C,c,C_1,c'$, etc., and may take different values in different parts of the paper. Given $a,b\in \mb R$, we will write $a\wedge b$ for $\min(a,b)$ and $a\vee b$ for 
$\max(a,b)$. For a positive integer $M$, $[M]$ denotes the set $\{1,\ldots,M\}$. 

We will frequently use the standard big-O and small-o notation for asymptotic relations between functions and sequences. Moreover, given two sequences $\{a_n\}_{n\geq 1}$ and $\{b_n\}_{n\geq 1}$ where $b_n\ne 0$ for all $n$, we will write that $a_n\ll b_n$ if $\frac{a_n}{b_n}=o(1)$ as $n\to\infty$. Note that $o(1)$ may denote different functions/sequences from line to line. 

For a function $f:\mb R\mapsto\mb R$, $f^{(m)}$ will denote its $m$-th derivative whenever it exists. Similarly, given $g:\mb R^d\mapsto\mb R$, $\partial_{x_j} g(x_1,\ldots,x_d)$ will stand for the partial derivative of $g$ with respect to the $j$-th variable. Finally, the sup-norm of $g$ is defined via $\|g\|_\infty:=\mathrm{ess \,sup}\{ |g(y)|: \, y\in \mb R^d\}$ and the convolution of $f$ and $g$ is denoted $f\ast g$. 

Given i.i.d. random variables $X_1,\ldots,X_N$ distributed according to $P$, $P_N:=\frac{1}{N}\sum_{j=1}^N \delta_{X_j}$ will stand for the associated empirical measure, where $\delta_{X}(f):=f(X)$. 
For a real-valued function $f$ and a signed measure $Q$, we will write $Qf$ for $\int f dQ$, assuming that the last integral is well-defined. Additional notation and auxiliary results will be introduced on demand.

\section{Optimal constants for the median of means estimator.}
\label{section:MOM-standard}
%####################################################

Recall that we are given an i.i.d. sample $X_1,\ldots,X_N$ from distribution $P$ with mean $\mu$ and variance $\sigma^2$. The median of means estimator of $\mu$ is constructed as follows: let $G_1\cup\ldots\cup G_k\subseteq [N]$ be an arbitrary (possibly random but independent from the data) collection of $k\leq N/2$ disjoint subsets (``blocks'') of cardinality $\lfloor N/k\rfloor$ each, $\bar X_j:=\frac{1}{|G_j|}\sum_{i\in G_j} X_i$ and 
\[
\wh\mu_{\mathrm{MOM}} = \med{\bar X_1,\ldots,\bar X_k}.
\] 
%$\wh\mu_{\mathrm{MOM}}$ is asymptotically normal: specifically,
It is known \citep[e.g.][]{lerasle2011robust,devroye2016sub} that $\wh\mu_{\mathrm{MOM}}$ satisfies inequality \eqref{eq:sg} for $t=k$ and $L= 8e^2$. This value of $L$ appears to be overly pessimistic however: it follows from Theorem 5 in \citep{minsker2017distributed} that if $k\to\infty$ sufficiently slow so that the bias of $\wh\mu_{\mathrm{MOM}}$ is of order $o(N^{-1/2})$, then 
\ben{
\label{eq:clt1}
\sqrt{N}(\wh\mu_{\mathrm{MOM}} - \mu)\xrightarrow{d} N\l(0, \frac{\pi}{2}\sigma^2\r)
}
as $k,N/k\to\infty$. In particular, if $\mb E|X|^{2+\delta}<\infty$ for some $0<\delta\leq 1$, then $k=o\l(N^{\delta/(1+\delta)}\r)$ suffices for the asymptotic unbiasedness and asymptotic normality to hold. 
%It also follows from Theorem 1 in \citep{minsker2017distributed} that 
%\ben{
%\label{eq:mom-1}
%\l| \wh\mu_{\mathrm{MOM}} - \mu \r| \leq \frac{3}{\sqrt 2}\sigma \l( \sqrt{\frac t N} + \sqrt{\frac{k}{N}}\cdot o\l(1\r)\r)}
%with probability at least $1-4e^{-t}$ whenever $0<t\leq c_1k$ and $k\leq c_2 N$ where $c_1$ is an absolute constant, $c_2$ may depend on $P$ and $o(1)\to 0$ as $m\to\infty$. 
Asymptotic relation \eqref{eq:clt1} suggests that the best value of the constant $L$ in the deviation inequality \eqref{eq:sg} for the estimator $\wh\mu_{\mathrm{MOM}}$ is $\pi+o(1)$. We will demonstrate that this is indeed the case. Denote 
%\[g(m):=g_P(m) = \sup_{t\in \mb R}\l| \pr{ \frac{\sum_{j=1}^m (X_j - \mu)}{\sigma\sqrt{m}}\leq t} - \Phi( t ) \r|,\]
\ben{
\label{eq:g}
g(m):= \frac{1}{\sqrt m} \mb E\l[ \l(\frac{X_1 - \mu}{\sigma}\r)^2 \min\l(\l|\frac{X_1-\mu}{\sigma}\r|, \sqrt m\r) \r],
}
Clearly, $g(m)\to 0$ as $m\to\infty$ for distributions with finite variance. \citet{feller1968berry} proved that $\sup_{t\in \mb R}\l| \Phi_m(t) - \Phi(t)\r|\leq 6g(m)$ where $\Phi_m$ and $\Phi$ are the distribution functions of $\frac{\sum_{j=1}^m X_j - \mu}{\sigma\sqrt{m}}$ and the standard normal law respectively. It is well known that $g(m) \leq C \mb E\l| \frac{X_1-\mu}{\sigma}\r|^q m^{-(q-2)/2}$ whenever $\mb E|X_1-\mu|^q<\infty$ for some $q\in(2,3]$.
The next result can be viewed as a non-asymptotic analogue of relation \eqref{eq:clt1}. 
%where $\Phi(t)$ is the cumulative distribution function of the standard normal distribution.
%############
\begin{theorem}
\label{th:dev-1}
The following bound holds:
\ben{
\label{eq:mom-subopt}
\pr{\l|\sqrt{N}(\wh\mu_{\mathrm{MOM}} - \mu)\r|\geq \sigma\sqrt t} \leq 2\exp{- \frac{t}{\pi}(1+o(1))}.
}
Here, $o(1)$ is a function that goes to $0$ as $k,N/k\to\infty$, uniformly over $t\in \l[ l_{k,N},u_{k,N} \r]$ for any sequences $l_{k,N}\gg k\,g^2(N/k)$ and $u_{k,N} \ll k$.
\end{theorem}
%#############
%$\frac{k\, g^2(m)}{l_{k,N}} \to 0$ and $\frac{u_{k,N}}{k } \to 0$
\begin{remark}
\begin{enumerate}
\item Note that the bound of the theorem holds in some range of the confidence parameter (such estimators are often called ``multiple-$\delta$'' in the literature, e.g., see \citet{devroye2016sub}), however, this range is distribution-dependent. In particular, if $\sqrt{k}\,g(N/k)\to 0$ as $k,N\to\infty$, the previous bound holds in the range $1\leq t \ll k$, but the function $g(\cdot)$ depends on $P$ and may converge to $0$ arbitrarily slow. Under additional assumptions, more concrete bounds can be deduced: for instance, if $\mb E|X/\sigma|^{2+\eps}<\infty$ for some $0<\eps\leq 1$, the condition $\sqrt{k}\,g(N/k)\to 0$ is satisfied if $k=o\l( N^{\frac{\eps}{1+\eps}} \r)$ as $N\to\infty$. In general, by choosing $k$ appropriately, we can construct a version of the median of means estimator that satisfies required guarantees for any $1\leq t \ll N$. 
\item The exact expression for the function $o(1)$ appearing in the statement of Theorems \ref{th:dev-1} and well as other results in the paper (e.g. Theorem \ref{th:U-mom}) is not made explicit. We remark that it depends on the distribution of $X_1$ through the function $g(\cdot)$ defined in \eqref{eq:g}, and on the ratios $\frac{kg^2(N/k)}{l_{k,N}}$ and $\frac{u_{k,N}}{k}$. 
\end{enumerate}
\end{remark}

%As we observed earlier, one of the advantages of the MOM estimator is its scale-invariance, in particular we may allow $\sigma^2$ to depend on $N$. If one is willing to give up this property (and, with it, much of the practical appeal of the estimator), then it is possible to get optimal constant $2$ instead of $\pi$ the tail bound assuming only the existence of a second moment. This can be achieved simply by using a Huber/Catoni-type estimator of location in place of the median. Such an estimator can be viewed as a ``hybrid'' between MOM and Catoni estimators, and we formally define and analyze it in Appendix \ref{sec:hybrid}.  

\begin{proof}[Proof of Theorem \ref{th:dev-1}]
As $\wh\mu_{\mathrm{MOM}}$ is scale-invariant, we can assume without loss of generality that 
$\sigma^2=1$. 
%and note that the general case follows by substituting $\frac{t}{\sigma^2}$ in place of $t$ in the final result.   
Denote $m=\lfloor N/k\rfloor$ for brevity, let $\rho(x) = |x|$, and note that the equivalent characterization of $\wh\mu_{\mathrm{MOM}}$ is 
\[
\wh\mu_{\mathrm{MOM}}\in \argmin_{z\in \mb R} \sum_{j=1}^k \rho\l( \sqrt{m}\l( \bar X_j - z\r)\r).
\]
The necessary conditions for the minimum of $F(z):=\sum_{j=1}^k \rho\l( \sqrt{m}\l( \bar X_j - z\r)\r)$ imply that $0\in \partial F(\wh\mu_{\mathrm{MOM}})$ -- the subgradient of $F$, hence the left derivative 
$F'_-(\wh\mu_{\mathrm{MOM}})\leq 0$. Therefore, if $\sqrt{N}\l(\wh\mu_{\mathrm{MOM}} - \mu\r)\geq \sqrt t$ for some $t > 0$, then $\wh\mu_{\mathrm{MOM}} \geq \mu + \sqrt{t/N}$ and, due to $F'_-$ being nondecreasing,  
$F'_-\l( \mu + \sqrt{t/N} \r)\leq 0$. It implies that 
\mln{
\label{eq:b001}
\pr{\sqrt{N}(\wh\mu_{\mathrm{MOM}} - \mu)\geq \sqrt t} \leq \pr{\sum_{j=1}^k \rho'_-\l( \sqrt{m}\l( \bar X_j - \mu - \sqrt{t/N}\r)\r) \geq 0} 
\\
= \pr{\frac{1}{\sqrt k}\sum_{j=1}^k \l(\rho'_-\l( \sqrt{m}\l( \bar X_j - \mu - \sqrt{t/N}\r)\r) - \mb E\rho'_-\r) \geq 
-\sqrt{k}\mb E\rho'_- }
}
where we used the shortcut $\mb E\rho'_-$ in place of $\mb E\rho'_-\l( \sqrt{m}\l( \bar X_j - \mu - \sqrt{t/N}\r)\r)$. 
Note that 
\mln{
\label{eq:b11}
-\sqrt{k} \mb E\rho'_-\l( \sqrt{m}\l( \bar X_j - \mu - \sqrt{t/N}\r)\r) = 
-\sqrt{k}\l(1 - 2 \pr{ \sqrt{m}\l( \bar X_j - \mu - \sqrt{t/N}\r) \leq 0} \r)
\\
= 2\sqrt{k}\l( \Phi\l( \frac{\sqrt t}{\sqrt{k}}\r) - \Phi(0) \r) - 2\sqrt{k}\l(\Phi\l( \frac{\sqrt t}{\sqrt{k}} \r) - \pr{ \sqrt{m}\l( \bar X_j - \mu \r) \leq \frac{\sqrt t}{\sqrt{k}}} \r)
\\
\leq 2\sqrt{k} \cdot g(m) + 2\sqrt t \frac{1}{\sqrt t/\sqrt{N/m}}\l( \Phi\l( \frac{\sqrt t}{\sqrt{N/m}}\r) - \Phi(0) \r). 
}
Since 
\mln{
\label{eq:b12}
2 \sqrt{t}\frac{1}{\sqrt t/\sqrt{N/m}}\l( \Phi\l( \frac{\sqrt t}{\sqrt{N/m}}\r) - \Phi(0) \r) = 2 \sqrt t \l( \phi(0) + O(t/\sqrt{N/m}) \r)
\\
= \sqrt t \l(\sqrt{\frac 2 \pi} + O(t/\sqrt{N/m}) \r)
}
where $\phi(t) = \Phi'(t)$, we see that 
\[
-\sqrt{k}\, \mb E\rho'_-\l( \sqrt{m}\l( \bar X_j - \mu - \sqrt{t/N}\r)\r) 
\leq 2\sqrt{k}\cdot g(m) + \sqrt t\l( \sqrt{\frac 2 \pi} + O(\sqrt{t/k}) \r)
\]
which is $\sqrt t\sqrt{\frac 2 \pi} \l(1 + o(1)\r)$ whenever $t\ll k$ and $t \gg k\, g^2(m)$. 
It remains to apply Bernstein's inequality to the right-hand side in \eqref{eq:b001}. Observe that 
\ml{
\var\l( \rho'_- \l( \sqrt{m}\l( \bar X_j - \mu - \sqrt{t/N}\r)\r) \r) = 4 \var\l( I \l\{ \sqrt{m}\l( \bar X_j - \mu \r) \leq  \sqrt{t/k} \r\} \r) 
\\
=4 \pr{\sqrt{m}\l( \bar X_j - \mu \r) \leq  \sqrt{t/k}}\l( 1 - \pr{\sqrt{m}\l( \bar X_j - \mu \r) \leq \sqrt{t/k}} \r) 
\leq 1, 
}
therefore
\ml{
\pr{\sqrt{N}(\wh\mu_{\mathrm{MOM}} - \mu)\geq \sqrt t} \leq 
\exp{-\frac{t}{\pi(1+o(1)) + \frac{2\sqrt t\sqrt{2\pi}}{3}\frac{1}{\sqrt k} \l(1 + o(1)\r) }}
\\
= \exp{- \frac{t}{\pi}(1+o(1))}
}
whenever $\sqrt{k}\,g(m) \ll t\ll \sqrt{k}$. 
Similar reasoning gives a matching bound for $\pr{\sqrt{N}(\wh\mu_{\mathrm{MOM}} - \mu)\leq -\sqrt t}$, and the result follows.
\end{proof}
One may ask whether the median of means estimator admits a more sample-efficient modification, one that would satisfy inequality \eqref{eq:sg} with a constant $L$ smaller than $\pi$. A natural idea is to require that the estimator is invariant with respect to permutations of the data or, equivalently, is a function of order statistics only. Such an extension of the MOM estimator was proposed by \citet{minsker2017distributed}, however no provable improvements for the performance over the standard MOM estimator were established rigorously. The question of such improvements, especially the guarantees expressed in the form \eqref{eq:sg}, is addressed next. 
%the improved performance guarantees have not been established rigorously. To this end, we will need to recall the notion of U-statistics, which is a focus of the next section.
Let us recall the proposed construction. Assume that $2\leq m < N$ and, given $J\subseteq [N]$ of cardinality $|J|=m$, set $\bar X_J:= \frac{1}{m}\sum_{j\in J}X_j$. Define $\m A_{N}^{(m)} = \l\{ J\subset [N]: \ |J|=m\r\}$ and
\ben{
\label{eq:U-mom-est}
\wh\mu_{N} : =\med{\bar X_J, \ J\in \m A_{N}^{(m)}},
}
where $\l\{\bar X_J, \ J\in \m A_{N}^{(m)}\r\}$ denotes the set of sample averages computed over all possible subsets of $[N]$ of cardinality $m$; in particular, unlike the standard median-of-means estimator, 
$\wh\mu_{N}$ is uniquely defined. 
Note that for $m=2$, $\wh\mu_N$ coincides with the well known Hodges-Lehmann estimator of location \citep{hodges1963estimates}. When $m$ is a fixed integer greater than $2$, $\wh\mu_N$ is known as the generalized Hodges-Lehmann estimator. Its asymptotic properties are well-understood and can be deduced from results by \citet{serfling1984generalized}, among other works. For example, its breakdown point is $1-(1/2)^{1/m}$ and, in case of normally distributed data, the asymptotic distribution of $\sqrt{N}(\wh\mu_N-\mu)$ is centered normal with variance $\Delta_m^2 = m\sigma^2\arctan\l( \frac{1}{\sqrt{m^2-1}}\r)$. In particular, $\Delta_m^2 = \sigma^2(1+o(1))$ as $m\to\infty$. When the underlying distribution is not symmetric however, $\wh\mu_N$ is biased for the mean, and the properties of this estimator in the regime $m\to\infty$ have not been investigated in the robust statistics literature (to the best of our knowledge). Only very recently, \cite{diciccio2022clt} proved that whenever $m\to\infty$, $m=o(\sqrt{N})$ and the sample is normally distributed, $\sqrt{N}(\wh\mu_N-\mu)\to N(0,\sigma^2)$. We will extend this result in several directions: first, by allowing a much wider class of underlying distributions, second, by including the case when $\sqrt{N} \ll m \ll N$ which is interesting as $\mathrm{bias}\l( \wh\mu_N\r)$ is $o\l( N^{-1/2}\r)$ in this regime, and finally by presenting sharp sub-Gaussian deviation inequalities for $\wh\mu_N$ that hold for heavy-tailed data.

Let us remark that an argument behind Theorem \ref{th:dev-1} combined with a version of Bernstein's inequality for U-statistics due to \cite{hoeffding1963probability} immediately implies that $\wh\mu_N$ satisfies relation \eqref{eq:mom-subopt}. Similar reasoning applies to other deviation guarantees for the classical median of means estimator that exist in the literature, so in this sense $\wh\mu_N$ always performs at least as good as $\wh\mu_{\mathrm{MOM}}$.

%Hoeffding's representation of the U-statistic as an average of independent random variables \citep[][section 5.1.6]{serfling2009approximation} can be used to show that all standard deviation guarantees for the classical median-of-means estimator, such as inequality \eqref{eq:mom-1} and the bound of Theorem \ref{th:dev-1}, also hold for its permutation-invariant version; see the proof of Theorem 5 in \citep{minsker2017distributed} for an example of such an argument. 
 Analysis of the estimator $\wh\mu_{N}$ is most naturally carried out using the language of U-statistics. The following section introduces the necessary background, while additional useful facts are summarized in section \ref{sec:tools}. 

%Specifically, it is a function of the vector of order statistics $X_{(1)},\ldots,X_{(N)}$ only which is complete, sufficient for the family of distributions $\m P_{2,\sigma}$. 
%It is therefore natural to expect that  $\wh\mu_{N}$ will have improved efficiency compared to 
%$\wh\mu_{\mathrm{MOM}}$ and that the deviation guarantees will hold with better constants. 
%It might still be somewhat surprising however that the resulting constants are essentially optimal: this is precisely the claim of the following two theorems.

%#####################################
\section{Asymptotic normality of U-statistics and the implications for $\wh\mu_N$.}
\label{sec:normality}
%#####################################

%In this section, we will establish sufficient conditions for asymptotic normality of non-degenerate U-statistics of growing order. 
Let $Y_1,\ldots,Y_N$ be i.i.d. random variables with distribution $P_Y$ and assume that $h_m:\mb R^m\mapsto \mb R,  \ m\geq 1$ are square-integrable with respect to $P_Y^{m}$ and permutation-symmetric functions, meaning that that 
$\mb E h_m^2(Y_1,\ldots,Y_m)<\infty$ and $h_m(x_{\pi(1)},\ldots,x_{\pi(m)}) = h_m(x_1,\ldots,x_m)$ for any $x_1,\ldots,x_m\in \mb R$ and any permutation $\pi:[m]\mapsto [m]$. Without loss of generality, we will also assume that 
$\mb E h_m:=\mb Eh_m(Y_1,\ldots,Y_m)=0$. Recall that $\m A_{N}^{(m)} = \l\{ J\subseteq [N]: \ |J|=m\r\}$. The U-statistic with kernel $h_m$ is defined as 
\ben{
\label{eq:U-stat1}
U_{N,m} = \frac{1}{{N\choose m}}\sum_{J\in \m A_{N}^{(m)}} h_m(Y_i, \ i\in J).
}
For $i\in[N]$, let 
\ben{
\label{eq:a00}
h_m^{(1)}(Y_i) = \mb E\l[ h_m(Y_1,\ldots,Y_{m})\,|\,Y_i\r].
%  = \mb E \l[U_{N,m} \,|\, Y_i \r]
}
%where $(\tilde Y_1,\ldots,\tilde Y_m)$ is an independent copy of $(Y_1,\ldots,Y_m)$, 
We will assume that $\pr{h_m^{(1)}(Y_1)\ne 0}>0$ for all $m$, meaning that the kernels $h_m$ are non-degenerate. The random variable 
\[
S_{N,m}:=\sum_{j=1}^N \mb E \l[U_{N,m} \,|\, Y_j \r] = \frac{m}{N} \sum_{j=1}^N h_m^{(1)}(Y_j),
\] 
known as the H\'{a}jek projection of $U_{N,m}$, is essentially the best approximation of $U_{N,m}$ in terms of the sum of i.i.d. random variables of the form $f(Y_1)+\ldots+f(Y_m)$. 
We are interested in the sufficient conditions guaranteeing that $\frac{U_{N,m} - S_{N,m}}{\sqrt{\var(S_{N,m})}} = o_P(1)$ as $N,m\to\infty$. Such asymptotic relation immediately implies that the limiting behavior of $U_{N,m}$ is defined by the H\'{a}jek projection $S_{N,m}$. 
Results of these type for U-statistics of fixed order $m$ are standard and well-known \citep{hoeffding1948class,serfling2009approximation,lee1990u}. However, we are interested in the situation when $m$ is allowed to grow with $N$, possibly up to the order $m=o(N)$. U-statistics of growing order were studied for example by \citet{frees1989infinite}, however existing results are not readily applicable in our framework. Very recently, such U-statistics have been investigated in relation to performance of Breiman's random forests algorithm (e.g. see the papers by \citet{song2019approximating} and \citet{peng2022rates}). The following theorem is essentially due to \cite{peng2022rates}; we give a different proof of this fact in Appendix \ref{proof:U-stat} as we rely on parts of the argument elsewhere in the paper.
%we will rely on MY PERFECT WONDERFUL BEAUTIFUL WIFE WHOM I ADORE AND CHERISH. 
%Existing results allowing $m$ to grow with $N$ that are readily applicable in our framework, for instance Theorem 3.1 in \citep{diciccio2020clt}, require that $m=o(\sqrt{N})$, among other conditions. However, we are interested in the situation when $m\gg \sqrt{N}$. The following result allows one to handle such cases.
%################
\begin{theorem}
\label{th:U-stat}
Assume that $\frac{\var\l( h_m(Y_1,\ldots,Y_m)\r)}{\var \l(h_m^{(1)}(Y_1)\r)} = o(N)$ as $N,m\to\infty$. \footnote{It is well known \citep{hoeffding1948class} that $\var \l(h^{(1)}(Y_1)\r) \leq \frac{\var(h_m)}{m}$, therefore the condition imposed on the ratio of variances implies that $m=o(N)$.} 
Then $\frac{U_{N,m} - S_{N,m}}{\sqrt{\var(S_{N,m})}} = o_P(1)$ as $N,m\to\infty$. 
\end{theorem}
%################
It is easy to see that asymptotic normality of $\frac{U_{N,m}}{\sqrt{\var(S_{N,m})}}$ immediately follows from the previous theorem whenever its assumptions are satisfied. Next, we will apply this result to establish asymptotic normality of the estimator $\wh\mu_N$ defined via \eqref{eq:U-mom-est}.

\begin{corollary}
\label{th:clt-1}
Let $X_1,\ldots,X_N$ be i.i.d. with finite variance $\sigma^2$. Moreover, assume that $\sqrt{\frac{N}{m}}\, g(m)\to 0$ as $N/m$ and $m\to\infty$. Then 
\[
\sqrt{N}\l( \wh\mu_N - \mu\r)\xrightarrow{d} \m N(0,\sigma^2)
\]
as $N/m$ and $m\to\infty$.
\end{corollary}
\begin{remark}
Requirement $\sqrt{\frac{N}{m}}\, g(m)\to 0$ guarantees that $\mathrm{bias}(\wh\mu_N)=o(N^{-1/2})$. Without this requirement, asymptotic normality can be established for the debiased estimator $\wh\mu_N - \mb E\wh\mu_N$. 
\end{remark}
\begin{proof}
%The argument proceeds in a similar fashion to the proof of Theorem \ref{th:dev-1}: 
Let $\rho(x)=|x|$ and note that the equivalent characterization of $\wh\mu_{N}$ is 
\[
\wh\mu_{N}\in \argmin_{z\in \mb R} \sum_{J\in \m A_N^{(m)}} \rho\l( \sqrt{m}\l( \bar X_J - z\r)\r).
\]
The necessary conditions for the minimum of this problem imply that for any fixed $t\geq 0$,
\aln{
\label{eq:c00}
&\pr{\sum_{J\in \m A_N^{(m)}} \rho'_-\l( \sqrt{m}\l( \bar X_J - \mu - tN^{-1/2}\r)\r) > 0} \leq
\pr{\sqrt{N}(\wh\mu_{N} - \mu)\geq t} \text{ and }
\\
\nonumber
&\pr{\sqrt{N}(\wh\mu_{N} - \mu)\geq t} \leq \pr{\sum_{J\in \m A_N^{(m)}} \rho'_-\l( \sqrt{m}\l( \bar X_J - \mu - tN^{-1/2}\r)\r) \geq 0}. 
}
Therefore, it suffices to show that the upper and lower bounds for $\pr{\sqrt{N}(\wh\mu_{N} - \mu)\geq t}$ converge to the same limit. To this end, we see that 
\mln{
\label{eq:b01}
\pr{\sum_{J\in \m A_N^{(m)}} \rho'_-\l( \sqrt{m}\l( \bar X_J - \mu - tN^{-1/2}\r)\r) \geq 0} 
\\
= \pr{\frac{\sqrt{N/m}}{{N \choose m}}\sum_{J\in \m A_N^{(m)}} \l(\rho'_-\l( \sqrt{m}\l( \bar X_J - \mu - tN^{-1/2}\r)\r) - \mb E\rho'_-\r) \geq -\sqrt{N/m}\mb E\rho'_- },
}
where $\mb E\rho'_-$ stands for $\mb E\rho'_-\l( \sqrt{m}\l( \bar X_J - \mu - tN^{-1/2}\r)\r)$. 
As it the proof of Theorem \ref{th:dev-1}, we deduce that 
$-\sqrt{N/m}\,\mb E\rho'_-\l( \sqrt{m}\l( \bar X_J - \mu - tN^{-1/2}\r)\r) \to \frac{t}{\sigma}\sqrt{\frac{2}{\pi}}$ whenever $\sqrt{N/m}\,g(m)\to 0$ and $N/m\to\infty$. 
It remains to analyze the U-statistic 
\[
\sqrt{\frac N m} \,U_{N,m} = \frac{\sqrt{N/m}}{{N \choose m}}\sum_{J\in \m A_N^{(m)}} \l(\rho'_-\l( \sqrt{m}\l( \bar X_J - \mu - tN^{-1/2}\r)\r) - \mb E\rho'_-\r).
\]
As the expression above is invariant with respect to the shift $X_j \mapsto X_j - \mu$, we can assume that $\mu=0$. %Moreover, we will also assume that $N=km$ (if $m$ does not divide $N$, one can use the fact that 
To complete the proof, we will verify the conditions of Theorem \ref{th:U-stat} allowing one to reduce the asymptotic behavior of $U_{N,m}$ to the analysis of sums of i.i.d. random variables. 
For $i \in [N]$, let 
%= \mb E \l[ \sqrt{\frac N m} U_{N,m} \,|\, X_i \r] 
\be{
h^{(1)}(X_i) 
= \sqrt{\frac{N}{m}} \,\mb E\l[ \rho'_-\l( \frac{1}{\sqrt m}\sum_{j=1}^{m-1} \tilde X_j + \frac{X_i}{\sqrt m} - t/\sqrt{N/m} \r)\,\big|\,X_i \r] - \sqrt{\frac{N}{m}}\mb E\rho'_-,
}
where $(\tilde X_1,\ldots,\tilde X_m)$ is an independent copy of $(X_1,\ldots,X_m)$. Our goal is to understand the size of $\var(h^{(1)}(X_1))$: specifically, we will show that 
$\var\l( \frac{m}{\sqrt N} h^{(1)}(X_1) \r) \to \frac{2}{\pi}$ as both $m$ and $N/m\to \infty$. 
Given an integer $l\geq 1$, let $\wt\Phi_{l}(t)$ be the cumulative distribution function of $\sum_{j=1}^l X_j$. Then 
\ml{
\frac{m}{\sqrt N} h^{(1)}(X_1) = \sqrt{m}\l(2\wt\Phi_{m-1}\l( \frac{tm}{\sqrt{N}} - X_1\r) -1 \r) - \sqrt{m} \mb E\,\rho'_- 
\\
= 2\sqrt{m}\l( \wt\Phi_{m-1}\l(\frac{tm}{\sqrt{N}} - X_1 \r)  - 
\mb E\wt\Phi_{m-1}\l( \frac{tm}{\sqrt{N}} - X_1\r) \r)
\\
=2\sqrt{m}\int_\mb R \l( \wt\Phi_{m-1}\l(\frac{tm}{\sqrt{N}} - X_1 \r)  - \wt\Phi_{m-1}\l( \frac{tm}{\sqrt{N}} - x \r) \r)dP(x).
}
We will apply the dominated convergence theorem to analyze this expression. Consider first the situation when the distribution of $X_1$ is non-lattice 
\footnote{We say that $X_1$ has lattice distribution if $\pr{X_1 \in \alpha+k\beta, \ k\in \mb Z} = 1$ and there is no arithmetic progression $A\subset Z$ such that $\pr{X_1 \in \alpha+k\beta, \ k\in A} = 1$}. Then the local limit theorem for non-lattice distributions \citep[][Theorem 2]{shepp1964local} implies that  
\[
\wt\Phi_{m-1}\l(a+h\r) - \wt\Phi_{m-1}\l( a \r) = \frac{h}{\sqrt{2\pi (m-1)}\sigma} \exp{-\frac{a^2}{2(m-1)\sigma^2}}+ o(m^{-1/2}),
\]
where $\sqrt{m}\cdot o(m^{-1/2})$ converges to $0$ as $m\to\infty$ for every $h$ and uniformly in $a$. Therefore, we see that conditionally on $X_1$ and for every $x$,
\mln{
\label{eq:c01}
\wt\Phi_{m-1}\l(\frac{tm}{\sqrt{N}} - x + (x - X_1) \r)  - \wt\Phi_{m-1}\l( \frac{tm}{\sqrt{N}} - x \r) 
 \\
 = \frac{x - X_1}{\sqrt{2\pi (m-1)}\sigma } \exp{-(tm/\sqrt{N}-x)^2/2(m-1)\sigma^2}  + o(m^{-1/2})
}
uniformly in $m$. Since $m=o(N)$ by assumption, $\exp{-(tm/\sqrt{N}-x)^2/2(m-1)\sigma^2} = 1 + o(1)$ as $m,N\to\infty$, hence 
\[
2\sqrt{m}\l( \wt\Phi_{m-1}\l(\frac{tm}{\sqrt{N}} - x + (x - X_1) \r)  - \wt\Phi_{m-1}\l( \frac{tm}{\sqrt{N}} - x \r) \r)
= 2\frac{x - X_1}{\sqrt{2\pi}\sigma } + o(1)
\]
$P$-almost everywhere. Next, we will show that $q_m(x,X_1):=\sqrt m\l( \wt\Phi_{m-1}\l(\frac{tm}{\sqrt{N}} - X_1 \r)  - \wt\Phi_{m-1}\l( \frac{tm}{\sqrt{N}} - x \r) \r)$ admits an integrable majorant that does not depend on $m$. 
Note that 
\[
|q_m(x,X_1)| \leq \sup_{z}\sqrt m\,\pr{\sum_{j=1}^{m-1} X_j \in \big(z, z+|x-X_1| \big]} 
\leq C |x-X_1|,
\]
where the last inequality follows from the well known bound for the concentration function (Theorem 2.20 in the book by \cite{petrov1995limit}); here, $C=C(P)>0$ is a constant that may depend on the distribution of $X_1$. 
We conclude that by the dominated convergence theorem, 
\[
\frac{m}{\sqrt N} h^{(1)}(X_1) \to \sqrt{\frac{2}{\pi}}\frac{X_1}{\sigma}
\]
as $m,N/m\to\infty$, $P$-almost everywhere. As 
\[
\l| \frac{m}{\sqrt N} h^{(1)}(X_1) \r|\leq 2\l| \int_\mb R q_m(x,X_1) dP(x) \r| \leq C\int_\mb R |x-X_1| dP(x)
\]
and $\mb E\l( \int_\mb R |x-X_1| dP(x) \r)^2<\infty,$ the second application of the dominated convergence theorem yields that $\var\l( \frac{m}{\sqrt N} h^{(1)}(X_1)\r) \to \var\l(\sqrt{\frac{2}{\pi}}\frac{X_1}{\sigma} \r) = \frac{2}{\pi}$ as $N/m$ and $m\to\infty$. 

It remains to consider the case when $X_1$ has a lattice distribution. In this case, a version of the local limit theorem \citep{petrov1995limit} states that 
\be{
\pr{\sum_{j=1}^{m-1} X_j = (m-1)\alpha + q\beta} = \frac{\beta}{\sqrt{2\pi(m-1)}\sigma} e^{-\frac{((m-1)\alpha + q\beta)^2}{2\sigma^2 (m-1)}}+o(m^{-1/2})
} 
where the $o(m^{-1/2})$ term is uniform in $q\in Z$. 
For any $y$ in the interval $\big( \frac{tm}{\sqrt{N}} - x, \frac{tm}{\sqrt{N}} - x + (x - X_1)  \big]$ of the form $y = (m-1)\alpha + q\beta$, we have that $e^{-\frac{y^2}{2\sigma^2 (m-1)}} = 1 + o(1)$ as $\frac m N \to 0$. Therefore, similarly to \eqref{eq:c01}, in this case
\be{
2\sqrt{m}\l(\wt\Phi_{m-1}\l(\frac{tm}{\sqrt{N}} - x + (x - X_1) \r)  - \wt\Phi_{m-1}\l( \frac{tm}{\sqrt{N}} - x \r) \r) 
= 2\frac{x-X_1}{\sqrt{2\pi} \sigma} + o(1)
}
$P$-almost everywhere, where we also used the fact that the number of points of the form $(m-1)\alpha+q\beta$ in the interval of interest equals $\frac{x - X_1}{\beta}$. The rest of the proof proceeds exactly as in the case of non-lattice distributions, and concludes the part of the argument related to $\var\l( \frac{m}{\sqrt N} h^{(1)}(X_1) \r)$.

To finish the proof, note that, since $\|\rho'_-\|_\infty=1$, $\var\l( \sqrt{N/m}\,\rho'_-\l( \sqrt{m}\l( \bar X_J - \mu - tN^{-1/2}\r)\r)\r)\leq \frac{N}{m}$, hence
\[
\frac{\var\l( \sqrt{N/m}\,\rho'_-\l( \sqrt{m}\l( \bar X_J - \mu - tN^{-1/2}\r)\r)\r)}{\var\l( h^{(1)}(X_1)\r)} 
\leq \frac{N/m}{\frac{2}{\pi}(1+o(1))N/m^2} = \frac{m}{\frac{2}{\pi}(1+o(1))} = o(N)
\]
as $m\to\infty$ and $N/m\to\infty$. 
Therefore, Theorem \ref{th:U-stat} applies and yields that 
\[
\frac{\sqrt{\frac N m}U_{N,m} - \frac{m}{N}\sum_{j=1}^N h^{(1)}(X_j)}{ \frac{m^2}{N} \var\l( h^{(1)}(X_j) \r)} = o_P(1),
\]
where $\frac{m^2}{N} \var\l( h^{(1)}(X_j) \r) = \frac{2}{\pi}(1+o(1))$. 
In view of the Central Limit Theorem, 
$\frac{m}{N}\sum_{j=1}^N h^{(1)}(X_j) \xrightarrow{d} N\l(0,\frac{2}{\pi}\r)$, and we conclude that 
$\sqrt{\frac{N}{m}} U_{N,m} \xrightarrow{d} N\l(0,\frac{2}{\pi}\r)$. Recalling \eqref{eq:b01}, we see that
\[
\pr{\sqrt{\frac{N}{m}} U_{N,m} \leq \sqrt{\frac N m} \mb E\rho'_-} \to 1 - \wt\Phi\l( \frac{t}{\sigma}\r),
\]
or $\limsup\limits_{m,N/m\to\infty}\pr{\sqrt{N}\l( \wh\mu_N - \mu\r)\geq t}\leq 1 - \wt\Phi\l( \frac{t}{\sigma}\r)$. 
Repeating the preceding argument for the lower bound for $\pr{\sqrt{N}\l( \wh\mu_N - \mu\r)\geq t}$, we get that $\liminf\limits_{m,N/m\to\infty}\pr{\sqrt{N}\l( \wh\mu_N - \mu\r)\geq t}\geq 1 - \wt\Phi\l( \frac{t}{\sigma}\r)$, whence the claim of the theorem follows.
\end{proof}
Corollary \ref{th:clt-1} implies that asymptotically, the estimator $\wh\mu_N$ improves upon 
$\wh\mu_{\mathrm{MOM}}$. The more interesting, and difficult, question is whether non-asymptotic sub-Gaussian deviation bounds for $\wh\mu_N$ with improved constant can be established, and to understand the range of the deviation parameter in which such bounds are valid.

%####################################################
\section{Deviation inequalities for U-statistics of growing order.}
\label{sec:growing-order}
%####################################################

%Result of Theorem \ref{th:clt-1} suggests that the estimator $\wh \mu_N$ admits optimal deviation bounds asymptotically. 
The ultimate goal of this section is to establish a non-asymptotic analogue of Corollary \ref{th:clt-1}. Recall that its proof relied on the classical strategy of showing that the higher-order terms in the Hoeffding decomposition of certain U-statistics are asymptotically negligible. To prove the desired non-asymptotic extension, one has be able to show that these higher-order terms are sufficiently small with exponentially high probability. However, classical tools used to prove such bounds rely on decoupling inequalities due to \cite{PM-decoupling-1995}. Unfortunately, the constants appearing in decoupling inequalities grow very fast with respect to the order $m$ of U-statistics, at least like $m^m$. As $m$ is allowed to grow with the sample size $N$ in our examples, such tools become insufficient to get the desired bounds in our framework. \citet{MA-ustat} derived an improved version of Bernstein's inequality for non-degenerate U-statistics where the sub-Gaussian deviations regime is controlled by $m\var(\h{1}_m(X))$ defined in equation \eqref{eq:a00}, rather than the larger quantity $\var(h_m)$ appearing in the inequality due to \citet{hoeffding1963probability}; however, this result is only useful when $m$ is essentially fixed. 
\citet{maurer2019bernstein} used different techniques that yield improvements over Arcones' result, in particular with respect to the order $m$; bounds obtained in this work are non-trivial for $m$ up to the order of $N^{1/3}$, however, this does not suffice for the applications required in the present paper. Moreover, unlike Theorem \ref{th:concentration} below, results in \citet{maurer2019bernstein} do not capture the correct behavior of degenerate U-statistics.
Recently, \citet{song2019approximating} made significant progress in studying U-statistics of growing order and developed tools that avoid using decoupling inequalities, however, their techniques apply when $m=o\l(\sqrt{N}\r)$, while we only require that $m=o(N)$. 

%h_m(x_1,\ldots,x_m):=
We will be interested in U-statistics with kernels of special structure that assumes ``weak'' dependence on each of the individual variables. Let the kernel be centered and written in the form $h_m\l(\frac{x_1}{\sqrt m},\ldots,\frac{x_m}{\sqrt m}\r)$, whence the corresponding U-statistic is
\ben{
\label{eq:U-stat1}
U_{N,m} = \frac{1}{{N\choose m}}\sum_{J\in \m A_{N}^{(m)}} h_m\l( \frac{X_i}{\sqrt{m}}, \ i\in J \r).
}
%Our goal is to obtain an exponentially decaying upper bound for the probability of deviations $\pr{\l| U_{N,m}\r|\geq t}$ that depends on the ``correct'' variance parameter. To this end, we are going to control each term in the Hoeffding decomposition of $U_{N,m}$. 
The Hoeffding decomposition of $U_{N,m}$ is defined as the sum
\ben{
\label{eq:f02}
U_{N,m} = \frac{m}{N}\sum_{j=1}^N h_m^{(1)}(X_j) + \sum_{j=2}^m \frac{{m\choose j} }{{N\choose j}} \sum_{J\in \m A_N^{(j)}} \h{j}_m(X_i,\, i\in J),
}
where $\h{j}_m(x_1,\ldots,x_j)=(\delta_{x_1} - P)\times\ldots\times(\delta_{x_j}-P)\times P^{m-j} h_m$ . We refer the reader to section \ref{sec:tools} where the Hoeffding decomposition and related background material is reviewed in more detail. 

We will assume that $U_{N,m}$ is non-degenerate, in particular, one can expect that the behavior of $U_{N,m}$ is determined by the first term $\frac{m}{N}\sum_{j=1}^N h_m^{(1)}(X_j)$ in the decomposition. 
In order to make this intuition rigorous, we need to prove that the higher-order terms are of smaller order with exponentially high probability. It is shown in the course of the proof of Theorem \ref{th:U-stat} that 
$\var\l(  \frac{{m\choose j} }{{N\choose j}} \sum_{J\in \m A_N^{(j)}} \h{j}_m(X_i,\, i\in J) \r)\leq \var(h_m) \l( \frac{m}{N}\r)^j$. 
%However, this bound is insufficient for our purposes for small and ``moderate'' values of $j$, and instead bounds for the moments of higher order for the terms of Hoeffding decomposition are required. 
However, to achieve our current goal, bounds for the moments of higher order are required.  
More specifically, the key technical difficulty lies in establishing the correct rate of decay of the higher moments with respect to the order $m$ of the U-statistic. We will show that under suitable assumptions, $\mb E^{1/q}\l|  \frac{{m\choose j} }{{N\choose j}} \sum_{J\in \m A_N^{(j)}} \h{j}_m(X_i,\, i\in J)\r|^q = O\l(j^{\eta_1}q^{\eta_2}\l( \frac{m}{N}\r)^{j/2}\r)$ for some $\eta_1>0$, $\eta_2>0$ and for all $q\geq 2$, $2\leq j\leq j_{\mathrm{\max}}$ for a sufficiently large $j_{\mathrm{max}}$. The crucial observation is that the upper bound for the higher-order $L_q$ norms is still proportional to $\l( \frac{m}{N}\r)^{j/2}$, same as the $L_2$ norm.  
The following result, essentially implied by the moment inequalities of this form, is a main technical novelty and a key ingredient needed to control large deviations of the higher order terms in the Hoeffding decomposition. 
\begin{theorem}
\label{th:concentration}
Let
\[
V_{N,j}= \frac{ {m\choose j}^{1/2} }{{N\choose j}^{1/2}} \sum_{J\in \m A_N^{(j)}} \h{j}_m\l(\frac{X_i}{\sqrt m},\, i\in J\r), \  f_j(x_1,\ldots,x_j)=\mb E h_m\l(\frac{x_1}{\sqrt m},\ldots,\frac{x_j}{\sqrt m},\frac{X_{j+1}}{\sqrt m},\ldots,\frac{X_m}{\sqrt m}\r)
\] 
and $\nu_k = \mb E^{1/k} |X_1-\mb EX_1|^k$. 
If the kernel $h_m$ is uniformly bounded, then there exists an absolute constant $c>0$ such that 
\be{
\pr{|V_{N.j}|\geq t} \leq \exp{ -\min\l( \frac{1}{c}\l(\frac{t^2}{\var(h_m)}\r)^{{\frac{1}{j}}},  \frac{\l( \frac{t}{\|h_m\|_\infty} \sqrt{\frac N j}\r)^{\frac{2}{j+1}}}{ c\l(m/j \r)^{\frac{j}{j+1}}} \r)}
}
whenever $\min\l( \frac{1}{c}\l(\frac{t^2}{\var(h_m)}\r)^{{\frac{1}{j}}}, \frac{\l( \frac{t}{\|h_m\|_\infty} \sqrt{\frac N j}\r)^{\frac{2}{j+1}}}{ c\l(m/j \r)^{\frac{j}{j+1}}} \r) \geq 2 $.
%\max\l(2, \frac{\l|\log(\var(h_m))\r|}{j}\r)$. 
Alternatively, suppose that 
\begin{enumerate}
\item[(i)] $\l\| \partial_{x_1}\ldots\partial_{x_j} f_j \r\|_\infty \leq \l(\frac{C_1(P)}{m}\r)^{j/2} j^{\gamma_1 j}$ for some $\gamma_1\geq \frac{1}{2}$;
\item[(ii)] $\nu_k \leq k^{\gamma_2}M$ for all integers $k\geq 2$ and some $\gamma_2\geq 0$, $M>0$.
\end{enumerate}
Then there exist constants $c_1(P),c_2(P)>0$ that depend on $\gamma_1$ and $\gamma_2$ only such that 
\ben{
\label{eq:concentration-2}
\pr{|V_{N,j}|\geq t}
\leq \exp{ -\min\l( \frac{1}{c_1}\l(\frac{t^2}{\var(h_m)}\r)^{{\frac{1}{j}}}, \l( \frac{t\sqrt{N/j}}{\l(c_2 M j^{\gamma_1-1/2}\r)^j}\r)^{\frac{2}{1+j(2\gamma_2+1)}} \r) }
}
whenever $\min\l( \frac{1}{c_1}\l(\frac{t^2}{\var(h_m)}\r)^{{\frac{1}{j}}}, \l( \frac{t\sqrt{N/j}}{\l(c_2 M j^{\gamma_1-1/2}\r)^j}\r)^{\frac{2}{1+j(2\gamma_2+1)}} \r) \geq  \max\l(2, \frac{\log(N/j)}{\gamma_2 j}\r)$. 
%In the course of the proof, we also deduce a version of the bound for smaller values of $t$.
\footnote{In the course of the proof, we show that whenever $\gamma_2=0$, corresponding to the case of a.s. bounded $X_1$, inequality \eqref{eq:concentration-2} is valid for all $t>0$.}
\end{theorem}
The proof of the theorem is given in section \ref{proof:concentration1}. Let us briefly discuss the imposed conditions. The first inequality requires only boundedness of the kernel and follows from a standard argument; it is mostly useful for the degenerate kernels of higher order $j$, for instance when $j \geq C m/\log(m)$). The main result is the second inequality of the theorem that provides a much better dependence of the tails on $m$ for small and moderate values of $j$. Assumption (ii) is a standard one: for instance, it holds with $\gamma_2=0$ for bounded random variables, $\gamma_2=1/2$ for sub-Gaussian and with $\gamma_2=1$ for sub-exponential random variables. As for assumption (i), suppose that the kernel $h_m$ is sufficiently smooth. In this case, 
\[
\partial_{x_1}\ldots\partial_{x_j} f_j(x_1,\ldots,x_j) = m^{-j/2} \mb E\l[ \l(\partial_{x_1}\ldots\partial_{x_j} h_m\r)\l(\frac{x_1}{\sqrt m},\ldots,\frac{x_j}{\sqrt m}, \frac{X_{j+1}}{\sqrt m},\ldots,\frac{X_m}{\sqrt m}\r)\r],
\] 
which is indeed of order $m^{-j/2}$ with respect to $m$. However, the functions $f_j$ are often smooth even if the kernel $h_m$ is not, as we will show later for the case of an indicator function (specifically, we will prove that required inequalities hold with $\gamma_1=\frac12$ for all $j\ll m/\log(m)$ under mild assumptions on the distribution of $X_1$). 
Next, we state a corollary -- a deviation inequality that takes a particularly simple form and suffices for most of the applications discussed later. It can be viewed as an extension of \citet{MA-ustat} version of Bernstein's inequality for the case of U-statistics of growing order. 
\begin{corollary}
\label{th:bernstein}
Suppose that 
\begin{enumerate}
\item[(i)] assumptions of Theorem \ref{th:concentration} hold for all $2\leq j\leq j_{\mathrm{max}}$ with $\gamma_1=\frac{1}{2}$;
\item[(ii)] the kernel $h_m$ is uniformly bounded;
%\item[(iii)] $m = o\l(N^{1-\delta}\r)$ for some $\delta>0$;
\item[(iii)] $\liminf_{m\to\infty}\var\l( \sqrt{m}\, h_m^{(1)}(X_1) \r)>0$;
\item[(iv)] $mM^2 = o\l(N^{1-\delta}\r)$ for some $\delta>0$.
% and $\limsup_{m\to\infty}\var\l( h_m(X_1) \r)<\infty$.
\end{enumerate}
Moreover, let $q(N,m)$ be an increasing function such that 
\[
q(N,m)=o\l(\min\l(\l(\frac{N}{mM^2}\r)^{\frac{1}{1+2\gamma_2}},j_{\mathrm{max}}\log(N/m)\r)\r) \text{ as } N/m\to\infty.
\]
Then for all $2\leq t\leq q(N,m)$,
\[
\pr{\l|U_{N,m}\r| \geq \sqrt{\frac{tm}{N}}} \leq 2\exp{-\frac{t}{2(1+o(1))\var\l( \sqrt{m}\, h_m^{(1)}(X_1) \r)}},
\] 
where $o(1)\to 0$ as $N/m\to \infty$ uniformly over $2\leq t\leq q(N,m)$. 
If $m = o\l( \frac{N^{1/2}}{\log(N)} \r)$, we can instead choose $q(N,m)$ such that $q(N,m)=o\l(\min\l(\l(\frac{N}{mM^2}\r)^{\frac{1}{1+2\gamma_2}}, \frac{Nj_{\mathrm{max}}}{m^2} \r)\r)$.
\end{corollary}
\begin{remark}
%This result essentially demonstrates sub-Gaussian behavior of $U_{N,m}$ in a certain range of the deviation parameter; 
The key point of the inequality is that the sub-Gaussian deviations are controlled by $\var\l( \sqrt{m}\, h_m^{(1)}(X_1) \r)$ rather than the sub-optimal quantity $\var(h_m)$ appearing in Hoeffding's version of Bernstein's inequality for U-statistics. Moreover, the range in which $U_{N,m}$ admits sub-Gaussian deviations is much wider compared to the implications of Arcones' inequality when $m$ is allowed to grow with $N$. Several comments regarding the additional assumptions are in order:
\begin{enumerate}
\item Assumption of uniform boundedness of the kernel $h_m$ is needed to ensure that we can apply Bernstein's concentration inequality to the first term of the Hoeffding decomposition. This suffices for our purposes but in general this condition can be relaxed.
\item Assumption on the asymptotic behavior of the variance is made to simplify the statement and the proof; if it does not hold, the result is still valid once the definition of $q(N,m)$ is modified to reflect the different behavior of the this quantity. 
We include the following heuristic argument which shows that $\lim_{m\to\infty}\var\l( \sqrt{m}\, h_m^{(1)}(X_1) \r)$ often admits a simple closed-form expression. Indeed, note that $\sqrt{m}\l( h_m^{(1)}(X_1) - h_m^{(1)}(0)\r) = \int_0^{X_1} \sqrt{m}\,\partial_{u} h_m^{(1)}(u) du$. If $\l\|\partial^2_{u} h_m^{(1)}\r\|_\infty = o(m^{-1/2})$, then 
\be{
\sqrt{m} \l| \partial_{u} h_m^{(1)}(u) - \partial_{u} h_m^{(1)}(0) \r| 
\leq \sqrt{m} \l\| \partial^2_{u} h_m^{(1)}\r\|_\infty u \to 0
} 
pointwise as $m\to\infty$. If the limit $\sqrt{m}\,\partial_{u} h_m^{(1)}(0)$ exists, then 
$
\sqrt{m}\l( h_m^{(1)}(X_1) - h_m^{(1)}(0)\r) \to \lim_{m\to\infty} \sqrt m \,\partial_u h_m^{(1)}(0) X_1
$, P-almost everywhere. Moreover, as $\sqrt m\| \partial_{u} h_m^{(1)}\|_\infty$ admits an upper bound independent of $m$ by assumption (i) of Theorem \ref{th:concentration} and $X_1$ is sufficiently integrable, Lebesgue's dominated convergence theorem applies and yields that $\var\l( \sqrt{m}\, h_m^{(1)}(X_1)\r) \to \l(\lim_{m\to\infty} \partial_u h_m^{(1)}(0)\r)^2 \var(X_1)$. For instance, this heuristic argument can often be made precise for kernels of the form $h\l( \sum_{j=1}^m \frac{x_j}{\sqrt m}\r)$. 
\item Finally, condition requiring that $mM^2 = o\l(N^{1-\delta}\r)$ is used to ensure that $\l(\frac{N}{mM^2}\r)^\tau \gg \log(m)$ for any fixed $\tau>0$ which simplifies the statement and the proof. 
\end{enumerate}
\end{remark}

\begin{proof}
	\begin{comment}

For the first part, it suffices to check that Theorem \ref{th:concentration} applies with with $t := \frac{\sqrt s}{j^2} \l( \frac{N}{m} \r)^{\frac{j-1}{2}}$, and that 
\[
\min\l( \frac{t^{\frac{2}{j}}}{c}, \l( \frac{t\sqrt{N/j}}{\l(c M j^{\gamma_1-1/2}\r)^j}\r)^{\frac{2}{1+j(2\gamma_2+1)}} \r) \gg s
\] 
under the stated assumptions. We omit simple algebraic calculations. To establish the second part

	\end{comment}
The union bound together with Hoeffding's decomposition entails that for any $t>0$ and $0<\eps<1$ (to be chosen later), 
\ml{
\pr{\l|U_{N,m}\r| \geq \sqrt{\frac{tm}{N}}} 
\\
\leq \pr{\l| \frac{m}{N}\sum_{j=1}^N h_m^{(1)}(X_j) \r|\geq (1-\eps) \sqrt t\sqrt{\frac{m}{N}}}
+ \pr{\l| \sum_{j=2}^m \frac{{m\choose j} }{{N\choose j}} \sum_{J\in \m A_N^{(j)}} \h{j}_m(X_i,\, i\in J) \r| \geq \eps\sqrt t\sqrt{\frac{m}{N}}}.
}
Bernstein's inequality yields that
\ml{
\pr{\l| \frac{m}{N}\sum_{j=1}^N h_m^{(1)}(X_j) \r|\geq (1-\eps)\sqrt t\sqrt{\frac{m}{N}}} 
\\
\leq 2\exp{-\frac{(1-\eps)^2 \,t/2}{ \var\l( \sqrt{m}\, h_m^{(1)}(X_1) \r) + (1-\eps)\frac13\sqrt{\frac{m}{N}}\|h_m\|_\infty t^{1/2}}} 
\\
=2\exp{-\frac{(1-\eps)^2\, t}{2\,\var\l( \sqrt{m}\, h_m^{(1)}(X_1) \r)(1+o(1))}}
}
where $o(1)\to 0$ as $N/m\to\infty$ uniformly over $s\leq q(N/m)$. 
It remains to control the expression involving higher order Hoeffding decomposition terms: specifically, we will show that under our assumptions, it is bounded from above by $\exp{-\frac{t}{2\,\var\l( \sqrt{m}\, h_m^{(1)}(X_1) \r)}} \cdot o(1)$ where $o(1)\to 0$ uniformly over the range of $t$. 
To this end, denote $t_\eps:=\eps^2 t$ and $j_\ast:=\min\l( j_{\mathrm{max}}, \lfloor \log(N/m) \rfloor +1\r)$. Observe that 
\mln{
\label{eq:e01}
\pr{\l| \sum_{j=2}^m \frac{{m\choose j} }{{N\choose j}} \sum_{J\in \m A_N^{(j)}} \h{j}_m(X_i,\, i\in J) \r| \geq \sqrt{t_\eps}\sqrt{\frac{m}{N}}} 
\\ \leq 
\pr{\l| \sum_{j=2}^{j_\ast} \frac{{m\choose j} }{{N\choose j}} \sum_{J\in \m A_N^{(j)}} \h{j}_m(X_i,\, i\in J) \r| \geq \frac{\sqrt{t_\eps}}{3}\sqrt{\frac{m}{N}}}
\\
+\pr{\l| \sum_{j=j_\ast+1}^{j_{\mathrm{max}}} \frac{{m\choose j} }{{N\choose j}} \sum_{J\in \m A_N^{(j)}} \h{j}_m(X_i,\, i\in J) \r| \geq \frac{\sqrt{t_\eps}}{3}\sqrt{\frac{m}{N}}}
\\
+\pr{\l| \sum_{j>j_{\mathrm{max}}} \frac{{m\choose j} }{{N\choose j}} \sum_{J\in \m A_N^{(j)}} \h{j}_m(X_i,\, i\in J) \r| \geq \frac{\sqrt{t_\eps}}{3}\sqrt{\frac{m}{N}}},
}
where the second sum may be empty depending on the value of $j_\ast$. 
First, we estimate the last term using Chebyshev's inequality: repeating the reasoning leading to equation \eqref{eq:a12} in the proof of Theorem \ref{th:U-stat}, we see that 
$\var\l(  \sum_{j>j_{\mathrm{max}}} \frac{{m\choose j} }{{N\choose j}} \sum_{J\in \m A_N^{(j)}} \h{j}_m(X_i,\, i\in J) \r) \leq \var(h_m)\l(\frac{m}{N}\r)^{ j_{\mathrm{max}}+1 } \l(1-m/N\r)^{-1}$, hence 
\ml{
\pr{\l| \sum_{j>j_{\mathrm{max}}} \frac{{m\choose j} }{{N\choose j}} \sum_{J\in \m A_N^{(j)}} \h{j}_m(X_i,\, i\in J) \r| \geq \frac{\sqrt{t_\eps}}{3}\sqrt{\frac{m}{N}}} \leq \frac{18\var(h_m)}{t_\eps}\l(\frac{m}{N}\r)^{ j_{\mathrm{max}}} 
\\
= 18\var(h_m) \exp{- j_{\mathrm{max}}\log(N/m) + \log(t_\eps) }
}
whenever $N/m\geq 2$. 
Alternatively, we can apply the first inequality of Theorem \ref{th:concentration} instead of Chebyshev's inequality to each term corresponding to $j>j_{\mathrm{max}}$ individually, with $t = t_{j,\eps}:= \frac{\sqrt{t_\eps}}{3 j^2} \l( \frac{N}{m} \r)^{\frac{j-1}{2}}$. It implies that 
\ml{
\pr{\l| \sum_{j>j_{\mathrm{max}}} \frac{{m\choose j} }{{N\choose j}} \sum_{J\in \m A_N^{(j)}} \h{j}_m(X_i,\, i\in J) \r| \geq \frac{\sqrt{t_\eps}}{3}\sqrt{\frac{m}{N}}}
\\ 
\leq \sum_{j>j_{\mathrm{max}}} \pr{\l| \frac{{m\choose j} }{{N\choose j}} \sum_{J\in \m A_N^{(j)}} \h{j}_m(X_i,\, i\in J) \r| \geq \frac{\sqrt{t_{j,\eps}}}{3}\sqrt{\frac{m}{N}}  }
\\
\leq m \max_{j>j_{\mathrm{max}}} \exp{-c\min\l(t_\eps^{1/j}\l( \frac Nm\r)^{\frac{j-1}{j}},\l( \frac{t_\eps}{\|h\|_\infty^2}\r)^{\frac{1}{j+1}}  \l( \frac{Nj}{m^2} \r)^{\frac{j}{j+1}} \r)}.
}
This bound is useful when $\l( \frac{Nj_{\mathrm{max}}}{m^2} \r)^{\frac{j_{\mathrm{max}}}{j_{\mathrm{max}}+1}} \gg j_{\mathrm{max}} \log(N/m)$, which is true whenever $m^2 \ll \frac{N}{\log^2(N)}$. If moreover $\eps \gg \frac{1}{\sqrt{\log(N)}}$, then the last probability is bounded from above by 
\[
\max_{j>j_{\mathrm{max}}} \exp{-c'\min\l(t_\eps^{1/j}\l( \frac Nm\r)^{\frac{j-1}{j}},\l( \frac{t_\eps}{\|h\|_\infty^2}\r)^{\frac{1}{j+1}}  \l( \frac{Nj}{m^2} \r)^{\frac{j}{j+1}} \r)}.
\]

To estimate the middle term (the probability involving the terms indexed by $j_\ast+1\leq j \leq j_{\mathrm{max}}$), we apply Theorem \ref{th:concentration} to each term individually for $t = t_{j,\eps}:= \frac{\sqrt{t_\eps}}{3 j^2} \l( \frac{N}{m} \r)^{\frac{j-1}{2}}$, keeping in mind that $\sum_{j\geq j_\ast+1} t_{j,\eps}\leq \frac{\pi^2}{18} \l( \frac{N}{m} \r)^{\frac{j-1}{2}}\sqrt{t_\eps}$. Note that for any $2\leq t \leq \frac{N}{m}$, $\eps > \frac{m}{N}$ and $j\geq \lfloor \log(N/m)\rfloor + 1$, 
\[
\min\l( \frac{t_{j,\eps}^{\frac{2}{j}}}{c}, \l( \frac{t_{j,\eps}\sqrt{N/j}}{\l(c M j^{\gamma_1-1/2}\r)^j}\r)^{\frac{2}{1+j(2\gamma_2+1)}} \r) \geq \frac{c_1}{M^{\frac{2}{1+2\gamma_2}}} \l( \frac{N}{m}\r)^{\frac{1}{1+2\gamma_2}},
\] 
whence 
\ml{
\pr{\l| \sum_{j=j_\ast+1}^{j_{\mathrm{max}}} \frac{{m\choose j} }{{N\choose j}} \sum_{J\in \m A_N^{(j)}} \h{j}_m(X_i,\, i\in J) \r| \geq \frac{\sqrt{t_\eps}}{3}\sqrt{\frac{m}{N}}} 
\\
\leq j_{\mathrm{max}} \exp{ -\frac{c_1}{M^{\frac{2}{1+2\gamma_2}}} \l( \frac{N}{m}\r)^{\frac{1}{1+2\gamma_2}} } 
\leq \exp{ -\frac{c_2}{M^{\frac{2}{1+2\gamma_2}}}\l( \frac{N}{m}\r)^{\frac{1}{1+2\gamma_2}} }.
}
Finally, to estimate the first term in the right side of inequality \eqref{eq:e01}, we again apply Theorem \ref{th:concentration}. With $t_{j,\eps}$ defined as above, 
\ml{
\pr{\l| \sum_{j=2}^{j_\ast} \frac{{m\choose j} }{{N\choose j}} \sum_{J\in \m A_N^{(j)}} \h{j}_m(X_i,\, i\in J) \r| \geq \frac{\sqrt{t_\eps}}{3}\sqrt{\frac{m}{N}}}
\\
\leq \sum_{j=2}^{j_\ast} \pr{\l| \frac{{m\choose j} }{{N\choose j}} \sum_{J\in \m A_N^{(j)}} \h{j}_m(X_i,\, i\in J) \r| \geq \frac{6}{\pi^2}\frac{\sqrt{t_{j,\eps}}}{3}\sqrt{\frac{m}{N}}}
\\
\leq  \sum_{j=2}^{j_\ast} \exp{- c\min\l(t^{1/j}_{\eps} \l( \frac N m\r)^{\frac{j-1}{j}}, \frac{t_{\eps}^{\frac{1}{1+j(1+2\gamma_2)}}}{M^{\frac{2j}{1+j(1+2\gamma_2)}}} \l( \frac N m\r)^{\frac{j}{1+j(1+2\gamma_2)}} \r)}
\\
\leq j_\ast \max_{2\leq j\leq j_\ast} \exp{- c\min\l(t^{1/j}_{\eps} \l( \frac N m\r)^{\frac{j-1}{j}}, \frac{t_{\eps}^{\frac{1}{1+j(1+2\gamma_2)}}}{M^{\frac{2j}{1+j(1+2\gamma_2)}}} \l( \frac N m\r)^{\frac{j}{1+j(1+2\gamma_2)}} \r)}.
}  
Whenever $\eps\geq \frac{1}{\sqrt{N/m)}}$, the last expression is upper bounded by 
\be{
\max_{2\leq j\leq j_\ast} \exp{- c_3\min\l(t^{1/j}_{\eps} \l( \frac N m\r)^{\frac{j-1}{j}}, \frac{t_{\eps}^{\frac{1}{1+j(1+2\gamma_2)}}}{M^{\frac{2j}{1+j(1+2\gamma_2)}}} \l( \frac N m\r)^{\frac{j}{1+j(1+2\gamma_2)}} \r)}
} 
for  $c_3$ small enough. Combining all the estimates, we obtain the inequality
\mln{
\label{eq:base-1}
\pr{\l| \sum_{j=2}^m \frac{{m\choose j} }{{N\choose j}} \sum_{J\in \m A_N^{(j)}} \h{j}_m(X_i,\, i\in J) \r| \geq \sqrt{t_\eps}\sqrt{\frac{m}{N}}} 
\leq 
\\  \max_{2\leq j\leq j_\ast} \exp{- c_3\min\l(t^{1/j}_{\eps} \l( \frac N m\r)^{\frac{j-1}{j}}, t_{\eps}^{\frac{1}{1+j(1+2\gamma_2)}} \l( \frac{N}{mM^2}\r)^{\frac{j}{1+j(1+2\gamma_2)}} \r)}
\\
+ \exp{ -c_2\l( \frac{N}{mM^2}\r)^{\frac{1}{1+2\gamma_2}} } + c_4\var(h_m) \exp{- j_{\mathrm{max}}\log(N/m) + \log(t_\eps) }
}
that holds if $\eps\geq \frac{1}{\sqrt{N/m}}$ and $2\leq t\leq \frac{N}{m}$. 
If $t<\l( \frac{N}{mM^2}\r)^{\frac{1}{1+2\gamma_2}} \eps^4$, then the first two terms on the right-hand side of the previous display are bounded by $e^{-\frac{ct_{\eps}}{\eps^3} } = e^{-\frac{ct}{\eps}}$ each, and if $t<\eps (j_{\mathrm{max}}-1) \log(N/m)$, the same is true for the last term. Therefore, 
if 
\[
t<\eps^4\min\l(  \l( \frac{N}{mM^2}\r)^{\frac{1}{1+2\gamma_2}}\,,(j_{\mathrm{max}}-1) \log(N/m)\r),
\] 
then 
\ml{
\pr{\l| \sum_{j=2}^m \frac{{m\choose j} }{{N\choose j}} \sum_{J\in \m A_N^{(j)}} \h{j}_m(X_i,\, i\in J) \r| \geq \sqrt{t_\eps}\sqrt{\frac{m}{N}}}
\\
\leq 3 \exp{-\frac{ct}{\eps}} = \exp{-\frac{t}{2\,\var\l( \sqrt{m}\, h_m^{(1)}(X_1) \r)}} \cdot o(1)
}
where the last equality holds whenever we choose $\eps:=\eps(N,m)$ such that $\eps(N,m)\to 0$ as $N/m\to\infty$. Specifically, take $\eps = \l( \frac{q(N,m)}{\min\l(  \l( \frac{N}{mM^2}\r)^{\frac{1}{1+2\gamma_2}}\,,j_{\mathrm{max}}\log(N/m)\r)} \r)^{1/4} $ where the function $q(N,m)$ was defined in the statement of the corollary, and conclusion follows immediately. 
If $m^2\ll \frac{N}{\log^2(N)}$, we can replace the last term in equation \eqref{eq:base-1} by 
\[
\max_{j>j_{\mathrm{max}}} \exp{-c'\min\l(t_\eps^{1/j}\l( \frac Nm\r)^{\frac{j-1}{j}},\l( \frac{t_\eps}{\|h\|_\infty^2}\r)^{\frac{1}{j+1}}  \l( \frac{Nj}{m^2} \r)^{\frac{j}{j+1}} \r)},
\]
which is bounded by $e^{-\frac{ct}{\eps}}$ whenever $t<\frac{Nj_{\mathrm{max}}}{m^2} \eps^4$. Final result in this case follows similarly.
\end{proof}

%############################################
\section{Implications for the median of means estimator.}
%############################################

We are going to apply results of the previous section to deduce non-asymptotic bounds for the permutation-invariant version of the median of means estimator. Recall that it was defined as 
%Recall its definition: let $m\geq 2$, and for any $J\subseteq [N]$ of cardinality $|J|=m$, set $\bar X_J:= \frac{1}{m}\sum_{j\in J}X_j$. We define
\be{
\wh\mu_{N} : =\med{\bar X_J, \ J\in \m A_{N}^{(m)}}.
}
\begin{theorem}
\label{th:U-mom}
Assume that $X_1,\ldots,X_N$ are i.i.d. copies of a random variable $X$ with mean $\mu$ and variance $\sigma^2$. Moreover, suppose that 
\begin{enumerate}
\item[(i)]  the distribution of $X_1$ is absolutely continuous with respect to the Lebesgue measure on $\mb R$ with density $\phi_1$;
\item[(ii)] the Fourier transform $\wh \phi_1$ of the density satisfies the inequality $\l| \wh\phi_1(x)\r| \leq \frac{C_1}{(1+|x|)^\delta}$ for some positive constants $C_1$ and $\delta$; 
\item[(iii)] $\mb E \l|(X_1-\mu)/\sigma \r|^{q}<\infty$ for some $\frac{3+\sqrt{5}}{2}<q \leq 3$;
%\item[(iv)] $m=O\l( N^{1-\eps}\r)$ for some $\eps>0$.
\end{enumerate}
Then the estimator $\wh\mu_{N}$ satisfies
\[
\pr{\l|\sqrt{N}(\wh\mu - \mu)\r|\geq \sigma\sqrt t} \leq 2\exp{-\frac{t}{2(1+o(1))}}
\]
where $o(1)\to 0$ as $m, \,N/m\to\infty$ uniformly for all $t\in \l[ l_{N,m},u_{N,m} \r]$ for any sequences $\{ l_{N,m}\}\,,\{u_{N,m}\}$ such that $l_{N,m}\gg \frac{N}{m^{q-1}}$ and $u_{N,m} \ll \frac{N}{m^{\frac{q}{q-1}} \vee m\log^2(N)}$.
%$\frac{N}{m^2}\ll \frac{t}{\sigma^2} \ll \frac{N}{m^{2-\frac{p}{3+p}} \vee \log^2(N)}$.
%$\frac{N}{m^{q-1}}\ll t \ll \frac{N}{m^{\frac{q}{q-1}} \vee m\log^2(N)}$
\end{theorem}
\begin{remark}
\begin{enumerate}
\item Let us recall the Riemann-Lebesgue lemma stating that $|\wh\phi_1(x)|\to 0$ as $|x|\to\infty$ for any absolutely continuous distribution, so assumption (ii) is rather mild;
\item The inequality $q>\frac{3+\sqrt{5}}{2}$ assures that $l_{N,m}$ and $u_{N,m}$ can be chosen such that $l_{N,m}\ll u_{N,m}$.
\end{enumerate}
\end{remark}
\begin{proof}
Throughout the course of the proof, we will assume without loss of generality that $\sigma^2=1$; general case follows by rescaling. Let us also recall that all asymptotic relations are defined in the limit as both $m$ and $N/m\to\infty$. 
Note that direct application of Corollary \ref{th:bernstein} requires existence of all moments of $X_1$, which is too prohibitive. Therefore, we will first show how to reduce the problem to the case of bounded random variables. 
Specifically, we want to truncate $X_j-\mu, \ j=1,\ldots,N$ in a way that preserves the decay rate of the characteristic function. To this end, let $R$ be a large constant (that will later be specified as an increasing function of $m$), and define the standard mollifier $\kappa(x)$ via 
$\kappa(x)=\begin{cases}
C_1\exp{-\frac{1}{1-x^2}}, & |x|<1, \\
0, & |x|\geq 1
\end{cases}$ where $C_1$ is chosen so that $\int_\mb R \kappa(x) = 1$. Moreover, let $\chi_R(x) = \l( I_{2R}\ast \kappa_R \r)(x)$ be the smooth approximation of the indicator function of the interval $[-2R,2R]$, where $I_{2R}(x) = I\{ |x|\leq 2R\}$ and $\kappa_R(x) = \frac{1}{R}\kappa(x/R)$; in particular, $\chi_R(x) = 1$ for $|x|\leq R$ and $\chi_R(x)=0$ for $|x|\geq 3R$. Set 
\[
\psi(x) = C_2 \phi_1(x+\mu) \chi_R(x)
\] 
where $C_2>0$ is such that $\int_\mb R \psi(x)dx = 1$. 
Suppose that $Y^{(R)}$ has distribution with density $\psi$ and note that by construction the laws of $X_1-\mu$ and $Y^{(R)}$, conditionally on the events $\{|X_1-\mu|\leq R\}$ and $\{|Y^{(R)}|\leq R\}$ respectively, coincide. Therefore, there exists a random variable $Z$ independent from $X_1$ such that 
\ben{
\label{eq:Y1}
Y_1^{(R)} := \begin{cases} X_1-\mu, & |X_1-\mu|\leq R, \\
Z, & |X_1-\mu|>R
\end{cases}
}
also has density $\psi$. Observe the following properties of $Y_1^{(R)}$: (a) $|Y_1^{(R)}|\leq 3R$ almost surely; (b) $\mb E h\l(Y_1^{(R)}\r)\leq C_2\mb Eh\l(X_1-\mu\r)$ for any nonnegative function $h$ -- indeed, this follows from the inequality $\psi(x)\leq C_2\phi_1(x+\mu)$; (c) $\l| \mb E Y_1^{(R)}\r|\leq (1+C_2)\frac{\mb E |X_1-\mu|^q I\{ |X_1-\mu |>R\}}{R^{q-1}}$. Indeed, 
\ml{
\l| \mb E Y_1^{(R)}\r| = \l| \mb E Y_1^{(R)} I\{ |X_1-\mu| \leq R \} + \mb E Y_1^{(R)} I\{ |X_1-\mu| > R \} \r| 
\\
= \l| \mb E (\mu - X_1)I\{ |X_1-\mu| > R \} + \mb E Y_1^{(R)} I\{ |Y_1^{(R)}| > R \}\r| 
\\
\leq \mb E \l| X_1-\mu \r| I\{ |X_1-\mu| > R \} + C_2 \mb E \l| X_1-\mu \r| I\{ |X_1-\mu| > R \}
}
where the last bound follows from property (b) for $h(x)=|x| I\{|x|>R\}$. It remains to apply H\"{o}lder's and Markov's inequalities. The final property of $Y_1^{(R)}$ is stated in a lemma below and is proven in the appendix.
\begin{lemma}
\label{lemma:truncation}
The characteristic function $\wh\psi(x)$ of $Y_1^{(R)}$ satisfies 
\[
\l| \wh\psi_1(x)\r|\leq \frac{C}{(1+|x|)^\delta}
\]
for all $x\in \mb R$ and a sufficiently large constant $C$.
\end{lemma}
\noindent Define $\rho(x)=|x|$. Proceeding as in the proof of Theorem \ref{th:dev-1}, we observe that 
\ben{
\label{eq:f01}
\pr{\sqrt{N}(\wh\mu - \mu)\geq \sqrt t} \leq \pr{ \frac{\sqrt{N/m}}{{N\choose m}}\sum_{J\in \m A_N^{(m)}} \rho_-'\l( \sqrt{m}\l( \bar X_J - \mu - \sqrt{t/N}\r)\r)\geq 0 }.
}
Our next goal is to show that for sufficiently large $R$, the U-statistic with kernel $\rho'_-$ appearing in \eqref{eq:f01} and evaluated at $X_1,\ldots,X_N$ can be replaced by the U-statistic evaluated over an i.i.d. sample $Y_1^{(R)},\ldots,Y_N^{(R)}$ where $Y_j^{(R)}$ is related to $X_j$ according to \eqref{eq:Y1}. To this end, recall that $\mb E|X_1-\mu|^{q}<\infty$, and choose $R$ as $R = c m^{\frac{1}{2(q-1)}}$ for some $c>0$. 
%Such that $\frac{1}{2(q-1)}<\tau<\frac{q-2}{2}$; note that the latter is only possible whenever $\frac{1}{2(q-1)}<\frac{q-2}{2}$ meaning that $q>\frac{3+\sqrt{5}}{2}$. 
%In this case we can set $\tau = \frac{1}{4}\l( q-2 + \frac{1}{q-1}\r)$. 
Next, observe that 
\mln{
\label{eq:f02}
\sum_{J\in \m A_N^{(m)}} \rho_-'\l( \sqrt{m}\l( \bar X_J - \mu - \sqrt{t/N}\r)\r) 
\\
= \sum_{J\in \m A_N^{(m)}} \l( \rho_-'\l( \sqrt{m}\l( \bar Y^{(R)}_J - \sqrt{t/N}\r)\r) - \mb E\rho_{-,R}' + \mb E\rho_{-}' \r) 
\\
+ \sum_{J\in \m A_N^{(m)}} \l( \rho_-'\l( \sqrt{m}\l( \bar X_J - \mu - \sqrt{t/N}\r)\r) - \rho_-'\l( \sqrt{m}\l( \bar Y^{(R)}_J - \sqrt{t/N}\r)\r) - \mb E\rho_{-}' + \mb E\rho_{-,R}' \r),
}
where $\mb E\rho_{-}' =  \rho_-'\l( \sqrt{m}\l( \bar X_J - \mu - \sqrt{t/N}\r)\r)$ and $\mb E\rho_{-,R}' = \mb E\rho_-'\l( \sqrt{m}\l( \bar Y^{(R)}_J - \sqrt{t/N}\r)\r)$. 
It was shown in the proof of Theorem \ref{th:dev-1} that 
\be{
\sqrt{N/m}\, \mb E  \rho_-' \leq 
C\sqrt{k}\cdot g(m) - \sqrt{t}\l( \sqrt{\frac 2 \pi} + O\l(\sqrt{\frac{t}{k}}\r) \r)
= -\sqrt{t} \sqrt{\frac 2 \pi} \l(1 + o(1)\r)
}
whenever $t\ll N/m$ and $t \gg \frac{N}{m}\, g^2(m)$. Let us remark that in view of imposed moment assumptions, $g(m) = O\l( m^{-(q-2)/2}\r)$. Moreover, it follows from Hoeffding's version of Bernstein's inequality for U-statistics \citep{hoeffding1963probability} that 
\ml{
 \frac{\sqrt{\frac Nm}}{{N\choose m}}\sum_{J\in \m A_N^{(m)}} \l( \rho_-'\l( \sqrt{m}\l( \bar X_J - \mu - \sqrt{t/N}\r)\r) - \rho_-'\l( \sqrt{m}\l( \bar Y^{(R)}_J - \sqrt{t/N}\r)\r) - \mb E\rho_{-}' + \mb E\rho_{-,R}' \r)
\\
\leq 2\mb E^{1/2}\l( \rho_-'\l( \sqrt{m}\l( \bar X_{[m]} - \mu - \sqrt{t/N}\r)\r) - \rho_-'\l( \sqrt{m}\l( \bar Y^{(R)}_{[m]} - \sqrt{t/N}\r)\r)\r)^2 \sqrt{s} \bigvee \frac{16s}{3}\sqrt{\frac m N}
}
with probability at least $1-e^{-s}$. We want to choose $s>0$ such that $t=o(s)$ and
\mln{
\label{eq:f021}
\alpha(s,R) := 2\mb E^{1/2}\l( \rho_-'\l( \sqrt{m}\l( \bar X_{[m]} - \mu - \sqrt{t/N}\r)\r) - \rho_-'\l( \sqrt{m}\l( \bar Y^{(R)}_{[m]} - \sqrt{t/N}\r)\r)\r)^2 \sqrt{s} \bigvee \frac{16s}{3}\sqrt{\frac m N} 
\\
=o\l(\sqrt{t}\r)
}
as $m,N/m \to\infty$. 
To estimate 
\[
\Sigma_m^2:=\mb E\l( \rho_-'\l( \sqrt{m}\l( \bar X_{[m]} - \mu - \sqrt{t/N}\r)\r) - \rho_-'\l( \sqrt{m}\l( \bar Y^{(R)}_{[m]} - \sqrt{t/N}\r)\r)\r)^2,
\] 
note that for any $a>0$, $\rho_-'\l( \sqrt{m}\l( \bar X_{[m]} - \mu - \sqrt{t/N}\r)\r) = \rho_-'\l( \sqrt{m}\l( \bar Y^{(R)}_{[m]} - \sqrt{t/N}\r)\r)$ whenever $\l|  \sqrt{m}\l( \bar Y^{(R)}_{[m]} - \sqrt{t/N}\r)\r|> a/2$, $\l| \sqrt{m}\l( \bar X_{[m]} - \mu - \sqrt{t/N}\r)\r|> a/2$ and $\l| \sqrt{m}\l( \bar X_{[m]} - \mu -  \bar Y^{(R)}_{[m]} \r) \r|\leq a$, hence 
\ml{
\Sigma_m^2 \leq 4\l( \pr{\l|  \sqrt{m}\l( \bar Y^{(R)}_{[m]} - \sqrt{t/N}\r)\r|\leq a}+\pr{\l| \sqrt{m}\l( \bar X_{[m]} - \mu - \sqrt{t/N}\r)\r|\leq a}\r) 
\\
+ 4 \pr{\l| \sqrt{m}\l( \bar X_{[m]} - \mu -  \bar Y^{(R)}_{[m]} \r) \r| > a}.
} 
Up to the additive error term $Cg(m)=O\l( m^{-(q-2)/2}\r)$, the distributions of $\sqrt{m}\bar X_{[m]}$ and $\sqrt{m}\bar Y^{(R)}_{[m]}$ can be approximated by the normal distribution, hence 
\[
\pr{\l|  \sqrt{m}\l( \bar Y^{(R)}_{[m]} - \sqrt{t/N}\r)\r|\leq a}+\pr{\l| \sqrt{m}\l( \bar X_{[m]} - \mu - \sqrt{t/N}\r)\r|\leq a}\leq C(a + g(m)). 
\]
Moreover, 
\mln{
\pr{\l| \sqrt{m}\l( \bar X_{[m]} - \mu -  \bar Y^{(R)}_{[m]} \r) \r| > a} 
= \pr{\l| \frac{1}{\sqrt m}\sum_{j=1}^m Y_j^{(R)} I\{ |Y_j^{(R)}| >R\} \r|\geq a }
\\
\leq \pr{\l| \frac{1}{\sqrt m}\sum_{j=1}^m Y_j^{(R)} I\{ |Y_j^{(R)}| >R\}  - \mb E\l( Y_j^{(R)} I\{ |Y_j^{(R)}| >R\} \r)\r|\geq a - \sqrt m \l|\mb EY_j^{(R)} I\{ |Y_j^{(R)}| >R\}\r|}
\\
\leq C_2\frac{\mb E |X_1-\mu|^2 I\{ |X_1-\mu |>R\}}{\l( a - C_2\sqrt{m}\l| \mb E(X_1-\mu)I\{|X_1-\mu|>R\} \r| \r)^2}
\\
\label{eq:f025}
\leq \frac{\mb E |X_1-\mu|^qI\{ |X_1-\mu |>R\}}{R^{q-2}\l( a -C_2\sqrt{m}\l| \mb E(X_1-\mu)I\{|X_1-\mu|>R\} \r| \r)^2} 
}
where we used property (b) of $Y_1^{(R)}$ along with H\"{o}lder's and Markov's inequalities. It is also clear that 
\be{
%\label{eq:f03}
\sqrt{m}\l|\mb E(X_1-\mu)I\{|X_1-\mu|>R\} \r| \leq \frac{\sqrt{m}\mb E |X_1-\mu|^q I\{ |X_1-\mu |>R\}}{R^{q-1}},
} 
therefore, for $R=cm^{\frac{1}{2(q-1)}}$ specified before, $\sqrt{m}\l|\mb E(X_1-\mu)I\{|X_1-\mu|>R\} \r|= o(1)$. 
%or $\tau_1 = \frac{1}{4}\l( q^2 - 3q +1 \r)>0$. 
Setting $a= 2C_2\frac{\sqrt{m}\mb E^{1/2} |X_1-\mu|^q I\{ |X_1-\mu |>R\}}{R^{q-1}}$, one easily checks that the right-hand side in \eqref{eq:f025} is at most $CR^{-(q-2)} = C' m^{-\frac{q-2}{2(q-1)}}$. 
whence $\Sigma^2_m = o(1)$. 
%O\l( m^{-\tau_1} + m^{-(q-2)/2} + m^{-\tau_2}\r) = O\l(m^{-\kappa}\r)
%for $\kappa = \min\l( \tau_1, (q-2)/2, \tau_2\r)$. 
Therefore, there exists a function $o(1)$ such that setting $s = t/o(1)$ yields the stated goal, namely, that $t=o(s)$ and $\alpha(s,R) = o(\sqrt t)$ where $\alpha(s,R)$ was defined in \eqref{eq:f021}. 
Combined with \eqref{eq:f02}, it implies that
\mln{
\pr{\sqrt{N}(\wh\mu - \mu)\geq \sqrt t} \leq o(1)\cdot e^{-t}
\\
\label{eq:U-bdd}
+ \pr{ \frac{\sqrt{\frac{N}{m}}}{{N\choose m}}\sum_{J\in \m A_N^{(m)}} \l(\rho_-'\l( \sqrt{m}\l( \bar Y^{(R)}_J -  \sqrt{t/N}\r)\r) - \mb E\rho'_{-,R}\r)\geq \sqrt{t}\sqrt{\frac{2}{\pi}}(1+o(1)) }. 
}
Note that the U-statistic in the display above is now a function of bounded random variables, hence we can apply Corollary \ref{th:bernstein} with $\gamma_2=0$. As $\|\rho'_-\|_\infty =1$, condition (ii) of the corollary holds. 
Let $\sqrt{\frac{m}{N}}\sum_{j=1}^N h^{(1)}(Y^{(R)}_{j})$ be the first term in Hoeffding decomposition of the U-statistic 
\[
\frac{\sqrt{N/m}}{{N\choose m}}\sum_{J\in \m A_{N}^{(m)}} \l(\rho_-'\l( \sqrt{m}\l( \bar Y^{(R)}_{J} - \mb E Y_1^{(R)} - \sqrt{t/N}+\mb E Y_1^{(R)}\r)\r) - \mb E \rho'_{-,R} \r).
\]
Following the lines of the proof of Theorem \ref{th:clt-1} and recalling that $\sqrt{m}\l| \mb EY_1^{(R)}\r|=o(1)$ in view of property (c) of $Y_1^{(R)}$ and the choice of $R$, we deduce that 
\be{
\var\l( \sqrt{m} h^{(1)}(Y^{(R)}_{1}) \r) = \frac{2}{\pi}(1+o(1))
}
where $o(1)\to 0$ as $m,N/m\to\infty$, validating assumption (iii) of the corollary. It remains to verify assumption (i) and specify the value of $j_{\max}$. 
Recall that $\rho'_-(x) = I\{ x\geq 0\} - I\{x<0\}$ and let $\wt Y_j^{(R)}$ stand for $Y_j^{(R)} - \mb EY_j^{(R)}$. The function $f_j(u_1,\ldots,u_j)$ appearing in the statement of Theorem \ref{th:concentration} can therefore be expressed as
\ml{
f_j(u_1,\ldots,u_j) = \mb E \rho'_-\l(  \frac{1}{\sqrt m}\sum_{i=1}^j u_i + \sqrt{\frac{m-j}{m}}\frac{\sum_{i=j+1}^m \wt Y_i^{(R)}}{\sqrt{m-j}} - \sqrt{\frac{tm}{N}} + \sqrt{m}\mb EY_1^{(R)}\r) 
\\
= \mb E \rho'_-\l(  \frac{1}{\sqrt m}\sum_{i=1}^j u_i + \sqrt{\frac{m-j}{m}}\frac{\sum_{i=j+1}^m \wt Y_i^{(R)}}{\sqrt{m-j}} - \sqrt{\frac{tm}{N}} + \sqrt{m}\mb EY_1^{(R)}\r) 
\\
= 2 \Phi_{m-j}\l(  \frac{1}{\sqrt{m-j}}\sum_{i=1}^j u_i - \sqrt{\frac{m}{m-j}}\l( \sqrt{\frac{tm}{N}}+\sqrt{m}\mb EY_1^{(R)}\r)\r) - 1
}
where for any integer $k\geq 1$, $\Phi_k$ stands for the cumulative distribution function of $\frac{1}{\sqrt k}\sum_{j=1}^k \wt Y_j^{(R)}$ and $\phi_k$ is the corresponding density function that exists by assumption. Consequently, 
\[
\partial_{u_j}\ldots\partial_{u_1} f_1(u_1,\ldots,u_j) = \frac{2}{(m-j)^{j/2}}\phi^{(j-1)}_{m-j}\l(  \frac{1}{\sqrt{m-j}}\sum_{i=1}^j u_i - \sqrt{\frac{m}{m-j}}\l( \sqrt{\frac{tm}{N}}+\sqrt{m}\mb EY_1^{(R)}\r)\r).
\]
The following lemma demonstrates that Theorem \ref{th:concentration} applies with $\gamma_1=1/2$ and that $j_{\max} = \frac{m}{\log(m)} \,o(1)$ in the statement of Corollary \ref{th:bernstein}. 
\begin{lemma}
\label{lemma:deriv-bound}
Let assumptions of Theorem \ref{th:U-mom} hold. Then for $m$ large enough and $j=o(m/\log m)$,
\be{
\l\| \phi_{m-j}^{(j-1)} \r\|_\infty \leq C \l(\frac{2j}{e}\r)^{j/2}
%\frac{C}{\sigma^{j}}
}
for a sufficiently large constant $C=C(P)$.
\end{lemma}
We postpone the proof of this lemma until section \ref{proof:deriv-bound}. 
As all the necessary conditions have been verified, the bound of Corollary \ref{th:bernstein} applies. Recalling that $t \gg \frac{N}{m}\, g^2(m)$ and that $g(m) \leq C \frac{\mb E|X_1-\mu|^q}{m^{(q-2)/2}}$, Corollary \ref{th:bernstein} yields that the probability in the right-hand side of inequality \eqref{eq:U-bdd} can be bounded from above by $\exp{-\frac{t}{2\sigma^2(1+o(1))}}$ for all 
\ben{
\label{eq:t-range}
\frac{N}{m^{q-1}} \ll t\leq q(N,m) 
}
whenever 
\[
q(N,m)=\min\l(\frac{N}{m R^2}, \frac{N}{m \log^2(N)}\r)\cdot o\l(1\r) \text{ as } N/m\to\infty.
\]
To get the expression for the second term in the minimum above from the bound of the corollary, it suffices to consider the cases when $m\geq \frac{\sqrt N}{\log(N)}o(1)$ and $m\leq \frac{\sqrt N}{\log(N)}o(1)$ separately; we omit the simple algebra. 
Since $R=c m^{\frac{1}{2(q-1)}}$, \eqref{eq:t-range} in only possible when $q-2 > \frac{1}{q-1}$, implying the requirement $q>\frac{3+\sqrt{5}}{2}$. 
%Moreover, due to the imposed lower bounds on $R$, it in fact suffices to assume that $q(N,m) = \frac{N}{m R^2} o(1)$. 
%Choosing $R^2 = m^{\frac{3}{q}} \l(\frac{\mb E |X_1-\mu|^q}{\mb E|X_1-\mu|^3} \r)^{2/q}$ and recalling that $q=3+p$, 
The final form of the bound stating that 
\[
\pr{\sqrt{N}(\wh\mu - \mu)\geq \sigma\sqrt t} \leq \exp{-\frac{t}{2(1+o(1))}}
\]
uniformly for all $\frac{N}{m^{q-1}}\ll t \ll \frac{N}{m^{\frac{q}{q-1}} \vee m\log^2(N)}$. The argument needed to estimate $\pr{\sqrt{N}(\wh\mu - \mu)\leq -\sigma\sqrt t}$ is identical.
\end{proof}

%###############################
\section{Open questions.}
\label{sec:discussion}
%###############################

Several potentially interesting questions and directions have not been addressed in this paper. We summarize few of them below. 
\begin{itemize}
\item[(i)] First is the question related to assumptions in Theorem \ref{th:U-mom}: does it still hold for distributions with only $2+\eps$ moments? And can the assumptions requiring absolute continuity and a bound on the rate of decay of the characteristic function be dropped? For example, Corollary \ref{th:clt-1} holds for lattice distributions as well. 
\item[(ii)] It is known that \citep{hanson1971bound} the sample mean based on i.i.d. observations from the multivariate normal distribution $N(\mu,\Sigma)$ satisfies the inequality 
\[
\l\| \bar X_N - \mu\r\|_2 \leq \sqrt{\frac{\mathrm{trace}(\Sigma)}{N}}+\sqrt{\frac{2t\|\Sigma\|}{N}}
\] 
with probability at least $1-e^{-t}$. Does there exist an estimator of the mean that achieves this bound (up to $o(1)$ factors) for the heavy-tailed distributions? 
%A natural candidate would be the tournaments-type estimator \citep[see][]{lugosi2019mean,lugosi2019near} that uses the univariate estimator $\wh\mu_n$ defined in \eqref{eq:U-mom-est} as a subroutine. 
Partial results in this direction have been recently obtained by \citet{lee2022optimal}.
\item[(iii)] Exact computation of the estimator $\wh\mu_N$ is infeasible, as it requires evaluation and sorting of $\asymp \l( \frac{N}{m}\r)^m$ sample means. Therefore, it is interesting to understand whether it can be replaced by $\med{\bar X_J, \ J\in \m B}$ where $\m B$ is a (deterministic or random) subset of $\m A_{N}^{(m)}$ of much smaller cardinality, while preserving the deviation guarantees. For instance, it is easy to deduce from results on incomplete U-statistics in section 4.3 of the book by \citet{lee1990u} combined with the proof of Corollary \ref{th:clt-1} that if $\m B$ consists of $M$ subsets selected at random with replacement from $\m A_N^{m}$, then the asymptotic distribution of $\sqrt{N}\l( \med{\bar X_J, \ J\in \m B} - \mu\r)$ is still $N(0,\sigma^2)$ as long as $M\gg N$. However, establishing results in spirit of Theorem \ref{th:U-mom} in this framework appears to be more difficult.
\end{itemize}
	
%\begin{enumerate}
%\item $\frac{N}{m^{q-1}}\geq 1$: as $t \gg \frac{N}{m^2}$ by assumption (c) above, we always have that $t\geq c\frac{N}{m^{q-1}}$, therefore we can choose $R$ as $R=\l(\frac{Nm}{N/m^2}\r)^{\frac{1}{2(3+p)}}$ so that $R^2 = m^{\frac{3}{3+p}}$, implying that 
%\[
%\pr{\sqrt{N}(\wh\mu - \mu)\geq \sqrt t} \leq \exp{-\frac{t}{\sigma^2(1+o(1))}}
%\]
%for all $\frac{N}{m^2}\ll t \ll \frac{N}{m^{2-\frac{p}{3+p}}}$.
%\item $\frac{N}{m^{q-1}} < 1$: 
%\end{enumerate}

%####################
\section{Remaining proofs.}
\label{sec:proofs}
%####################

The proofs omitted in the main text are presented in this section. 

%#####################
\subsection{Technical tools.}
\label{sec:tools}
%#####################

Let us recall the definition of Hoeffding's decomposition \citep{hoeffding1948class} and closely related concepts that are at the core of many arguments related to U-statistics. Assume that $Y_1,\ldots,Y_N$ are i.i.d. random variables with distribution $P_Y$. 
% and $h:\mb R^m\mapsto \mb R,  \ m\geq 1$ is a square-integrable with respect to $P_Y^{m}$ and permutation-symmetric function such that $\mb E h_m:=\mb Eh_m(Y_1,\ldots,Y_m)=0$. 
Recall that $\m A_{N}^{(m)} = \l\{ J\subseteq [N]: \ |J|=m\r\}$ and that the U-statistic with permutation-symmetric kernel $h_m$ is defined as 
%\label{eq:U-stat1}
\[
U_{N,m} = \frac{1}{{N\choose m}}\sum_{J\in \m A_{N}^{(m)}} h_m(Y_i, \ i\in J),
\] 
where we assume that $\mb Eh_m=0$. Moreover, for $j=1,\ldots,m,$ define the projections
\ben{
\label{eq:pi}
(\pi_j h_m)(y_1,\ldots,y_j):=(\delta_{y_1} - P_Y)\times\ldots\times(\delta_{y_j}-P_Y)\times P_Y^{m-j} h_m.
}
For brevity and to ease notation, we will often write $h_m^{(j)}$ in place of $\pi_j h_m$. The variances of these projections will be denoted by
\[
\delta_j^2 :=\var\l( \h{j}_m(Y_1,\ldots,Y_j) \r).
\]
In particular, $\delta_m^2 = \var(h_m)$. It is well known \citep{lee1990u} that $\h{j}_m$ can be viewed geometrically as orthogonal projections of $h_m$ onto a particular subspace of $L_2(P_Y^{m})$. 
The kernels $\h{j}_m$ have the property of complete degeneracy, meaning that $\mb E \h{j}_m(y_1,\ldots,y_{j-1},Y_j) = 0$ for $P_Y$-almost all $y_1,\ldots,y_{j-1}$ while $\h{j}_m(Y_1,\ldots,Y_j)$ is non-zero with positive probability. One can easily check that $h(y_1,\ldots,y_m) = \sum_{j=1}^m \sum_{J\subseteq [m]: |J|=j} \h{j}_m(y_i,\, i\in J)$, in particular, the partial sum $\sum_{j=1}^k \sum_{J\subseteq [m]: |J|=j} \h{j}_m(y_i,\, i\in J)$ is the best approximation of $h_m$, in the mean-squared sense, in terms of sums of functions of at most $k$ variables. 
The Hoeffding decomposition states that (see \citep{hoeffding1948class} as well as the book by \citet{lee1990u})
\ben{
\label{eq:hoeffding}
U_{N,m} = \sum_{j=1}^m {m\choose j} U_{N,m}^{(j)},
}
where $U_{N,m}^{(j)}$ are U-statistics with kernels $\h{j}_m$, namely 
$U_{N,m}^{(j)}:=\frac{1}{{N\choose j}} \sum\limits_{J\in \m A_{N}^{(j)}} \h{j}_m(Y_i, \ i\in J)$. Moreover, all terms in representation \eqref{eq:hoeffding} are uncorrelated. 
%Finally, we often use the shortcut $\h{j}$ in place of $\pi_j h$. 

	\begin{comment}

When dealing with U-statistics, we are often interested in the exponential moments of the form 
$\mb E \exp{|Z|^{\alpha}}$ for some random variable $Z$ and $0<\alpha<1$. As the function $\psi_\alpha(z) = \exp{|z|^\alpha}$ is only convex in the region $\{ |z|^\alpha\geq \frac{1-\alpha}{\alpha} \}$, we will often use its convex modification 
\ben{
\label{eq:psi}
\tilde \psi_\alpha(z) = \begin{cases} e^{|z|^\alpha}, &  |z|^\alpha > \frac{1-\alpha}{\alpha}, \\
1 + \frac{e^{\frac{1-\alpha}{\alpha}}-1}{\l( \frac{1-\alpha}{\alpha}\r)^{1/\alpha}} |z|, & |z|^\alpha \leq \frac{1-\alpha}{\alpha}. \end{cases}
}
It is also easy to see that 
\be{
\label{eq:psi2}
\tilde\psi_\alpha(z) \leq \psi_\alpha(z) \leq \min\l( \tilde \psi_\alpha(z) + e^{\frac{1-\alpha}{\alpha}}, \tilde \psi_\alpha(z) e^{\frac{1-\alpha}{\alpha}}\r).
} 

	\end{comment}
%Informally speaking, one may think of the partial sums $\sum_{j=1}^k {m\choose j} U_{N,m}^{(j)}$ as the best approximation, in $L_2$ sense, of the U-statistic $U_{N,m}$ in terms of U-statistic of order $k$.
%We recall some of the basic facts and existing results that our proofs often rely upon. Given a metric space $(T,\rho)$, the covering number $N(T,\rho,\eps)$ is defined as the smallest $N\in \mb N$ such that there exists a subset 
%$F\subseteq T$ of cardinality $N$ with the property that for all $z\in T$, $\rho(z,F)\leq \eps$. 
%When metric $\rho$ is clear from the context, we will simply write $N(T,\eps)$. 

Next, we recall some useful moment bounds, found for instance in the book by \citet{Decoupling}, for the Rademacher chaos variables. Let $\eps_1,\ldots,\eps_N$ be i.i.d. Rademacher random variables (random signs), $\{ a_J, \ J\in \m A_N^{(l)} \} \subset \mb R$, and $Z = \sum_{J\in \m A_{N}^{(l)} } a_J \prod_{i\in J} \eps_i$. Here, $\prod_{i\in J} \eps_i = \eps_{i_1}\cdot\ldots\cdot \eps_{i_l}$ for $J=\{i_1,\ldots,i_l\}$. 
\begin{fact}[Bonami inequality]
\label{fact:1}
Let $\sigma^2(Z) = \var(Z) = \sum_{J\in \m A_{N}^{(l)}} a_J^2$. Then for any $q>2$,
\[
\mb E |Z|^q \leq \l( q-1\r)^{ql/2} \l(\sigma^2(Z) \r)^{q/2}.
\]
%\[
%\mb E \exp{\l( \frac{|Z|}{4e\sigma(Z)} \r)^{2/l}}\leq 2 \text{ and }
%\mb E\exp{\lambda |Z|^{\alpha}} \leq 2\exp{K(\alpha,j) \l(\sigma^\alpha(Z) \lambda \r)^{2/(2-\alpha j)}}
%\]
%for any $\lambda>0$, any integer $l\geq 2$ and any $\alpha<\frac{2}{j}$, where $K(\alpha,j) = 
%\l( 1 - \frac{\alpha j}{2}\r)\l( (4e)^{\alpha} (\alpha j/2)^{\alpha j/2} \r)^{\frac{2}{2-\alpha j}}$.
\end{fact}
Now we state a version of the symmetrization inequality for completely degenerate U-statistics due to \citet{sherman1994maximal}, also see the paper by \citet{song2019approximating} for the modern exposition of the proof. The main feature of this inequality, put forward by \citet{song2019approximating}, is the fact that its proof does not rely on decoupling, and yields constants that do not grow too fast with the order of U-statistics. 
\begin{fact}
\label{fact:2}
Let $h$ be a completely degenerate kernel of order $l$, and $\Phi$ -- a convex, nonnegative, non-decreasing function. Moreover, assume that $\eps_1,\ldots,\eps_N$ are i.i.d. Rademacher random variables. Then 
\[
\mb E\Phi\l( \sum_{ 1\leq j_1<\ldots<j_l\leq N} h(Y_{j_1},\ldots,Y_{j_l})\r) \leq \mb E\Phi\l(2^l \sum_{1\leq j_1<\ldots<j_l\leq N} \eps_{j_1}\ldots\eps_{j_l} h(Y_{j_1},\ldots,Y_{j_l}) \r). 
\]
\end{fact}
Next is the well-known identity, due to \citet{hoeffding1963probability}, that allows to reduce many problems for non-degenerate U-statistics to the corresponding problems for the sums of i.i.d. random variables.
\begin{fact}
\label{fact:3}
The following representation holds: 	
\[
U_{N,m} = \frac{1}{N!} \sum_{\pi} W_{\pi},
\]
where the sum is over all permutations $\pi:[N]\mapsto[N]$, and 
\[
W_{\pi} = \frac{1}{k}\l( h_m\l(Y_{\pi(1)},Y_{\pi(2)},\ldots,Y_{\pi(m)}\r) + \ldots +  h_m\l(Y_{\pi((k-1)m+1)},Y_{\pi((k-1)m+2)},\ldots,Y_{\pi(km)}\r)\r)
\]
for $k=\lfloor N/m \rfloor$.
\end{fact}
Finally, we state a version of Rosenthal's inequality for the moments of sums of independent, nonnegative random variables with explicit constants, see \citep{boucheron2013concentration,chen2012masked}. 
\begin{fact}
\label{fact:rosenthal}
Let $Y_1,\ldots,Y_N$ be independent random variables such that $Y_j\geq 0$ with probability $1$ for all $j\in [N]$. Then for any $q\geq 1$,
\[
\l(\mb E\l| \sum_{j=1}^N Y_j \r|^q\r)^{1/q}\leq \l( \l(\sum_{j=1}^N \mb E Y_j\r)^{1/2} + 2\sqrt{eq} \l( \mb E\max_{j=1,\ldots,N}Y_j^{q}\r)^{1/2q}\r)^2.
\]
\end{fact}

	\begin{comment}

Finally, we recall the result due to \citet{bourgain1979walsh} and its version with better constants due to \citet{kwapien2010hoeffding} regarding the continuity of the operators $\pi_j$ in $L_p$ (we will only state the version for $p>2$ but the general result holds for $p>1$).
\begin{fact}
\label{fact:4}
Let $h:\mb R^m\mapsto \mb R$ be in $L_p(P^m)$ for some $p>2$. Then  
\[
\mb E^{1/p} \l| \sum_{|J|=j}(\pi_j h)(X_i, \ i\in J)\r|^p \leq \l(C \frac{p}{\log p}\r)^j \mb E^{1/p} |h(X_1,\ldots,X_m)|^p,
\]
where $C>0$ is an absolute constant.
\end{fact}
The Hermite polynomials $H_j(x), \ j \geq 0$ are defined via $H_j(x)=(-1)^j e^{x^2} \frac{d^j}{dx^j}e^{-x^2}$, while the closely related Hermite functions are $\psi_j(x) = (2^j j!)^{-1/2} \pi^{-1/4} e^{-x^2/2}H_j(x)$. 
\begin{fact}[Cram\'{e}r's inequality, see \cite{indritz1961inequality}]
\label{fact:5}
The following inequality holds: 
\[
\sup_{x\in \mb R} |\psi_j(x)| \leq \pi^{-1/4}.
\] 
\end{fact}
\begin{fact}
\label{fact:6}
Let $\xi_j, \ j=1,\ldots,N$ be i.i.d. centered random variables such that $|\xi_1|\leq M$ almost surely and $\var(\xi_1) = \sigma^2$. Then 
\[
\mb E\exp{t\sum_{j=1}^N \xi_j}\leq \exp{\frac{N\sigma^2 t^2}{2- \frac{2}{3}Mt}}
\]
for $t<\frac{3}{M}$.
\end{fact}

	\end{comment}

%################################
\subsection{Proof of Theorem \ref{th:U-stat}.}
\label{proof:U-stat}
%################################

Recall that 
\al{
\h{j}_m(y_1,\ldots,y_j)&:=(\delta_{y_1} - P_Y)\times\ldots\times(\delta_{y_j}-P_Y)\times P_Y^{m-j} h_m, 
\\
\delta_j^2 &:=\var\l( \h{j}_m(Y_1,\ldots,Y_j) \r).
}
%We will also write $\var(h_m)$ for $\delta_m^2 = \var\l( h_m(Y_1,\ldots,Y_m) \r)$ for brevity. 
%It is well known \citep{lee1990u} that $\h{j}_m$ can be viewed as a projection of $h_m$ onto a particular subspace of $L_2(P_Y^{m})$. 
It is easy to verify that 
\ml{
h_m(Y_1,\ldots,Y_m) = (\delta_{Y_1} - P_Y + P_Y)\times\ldots\times(\delta_{Y_m}- P_Y + P_Y) h_m
= \sum_{j=1}^m \sum_{J\subseteq [m]: |J|=j} \h{j}_m(Y_i,\, i\in J)
}
and that the terms in the sum above are mutually orthogonal, yielding that 
\ben{
\label{eq:a01}
\var\l(h_m(Y_1,\ldots,Y_m)\r) = \sum_{j=1}^m {m\choose j} \delta_j^2.
}
Moreover, as a corollary of Hoeffding's decomposition, one can get the well known identities
\al{
&\var(U_{N,m}) = \sum_{j=1}^m \frac{{m\choose j}^2}{{N\choose j}}\delta_j^2, 
\quad\var(S_{N,m}) = \frac{m^2}{N}\delta_1^2,
\\
&\var(U_{N,m} - S_{N,m}) = \var(U_{N,m}) - \var(S_{N,m}) = \sum_{j=2}^m \frac{{m\choose j}^2}{{N\choose j}}\delta_j^2.
}
See Chapters 1.6 and 1.7 in the book by \cite{lee1990u} for detailed derivations of these facts. 
The simple but key observation following from equation \eqref{eq:a01} is that for any $j\in[m]$, $\var(h_m)\geq {m\choose j}\delta_j^2$, or
\ben{
\label{eq:a02}
\delta_j^2\leq \frac{\var(h_m)}{{m\choose j}}.
}
Therefore, 
\mln{
\label{eq:a12}
\var(U_{N,m} - S_{N,m}) = \sum_{j=2}^m \frac{{m\choose j}^2}{{N\choose j}}\delta_j^2 
\leq \var(h)\sum_{j=2}^m \frac{{m\choose j}}{{N\choose j}} \leq \var(h)\sum_{j\geq 2} \l( \frac{m}{N}\r)^j
\\
=\var(h)\l(\frac{m}{N}\r)^2 \l(1-m/N\r)^{-1},
}
where we used the fact that $\frac{{m\choose j}}{{N\choose j}}\leq \l( \frac m N\r)^j$ for $m\leq N$: indeed, the latter easily follows from the identity $\frac{{m\choose j}}{{N\choose j}} = \frac{m(m-1)\ldots(m-j+1)}{N(N-1)\ldots (N-j+1)}$. It is well known \citep{hoeffding1948class} that $\var \l(h^{(1)}(Y_1)\r) \leq \frac{\var(h_m)}{m}$, therefore the condition $\frac{\var\l( h_m(Y_1,\ldots,Y_m)\r)}{\var \l(h_m^{(1)}(Y_1)\r)} = o(N)$ imposed on the ratio of variances implies that $m=o(N)$. Therefore, for $m,N$ large enough (so that $m/N\leq 1/2$), 
\be{
\frac{\var(U_{N,m} - S_{N,m})}{\var(S_{N,m})} \leq 2 \frac{\var(h_m)\l(\frac{m}{N}\r)^2 }{\delta_1^2 m^2/N} 
= 2 \frac{\var(h_m)}{N\delta_1^2} = o(1)
}
by assumption, yielding that $\frac{U_{N,m} - S_N}{\var^{1/2}(S_N)} = o_P(1)$ as $N,m\to\infty$.

%########################################
\subsection{Proof of Theorem \ref{th:concentration}.}
\label{proof:concentration1}
%########################################

We are going to estimate $\mb E|V_{N,j}|^q$ for an arbitrary $q>2$. It follows from the symmetrization inequality (Fact \ref{fact:2}) followed by the moment bound stated in Fact \ref{fact:1} that 
\ml{
\mb E|V_{N,j}|^q \leq 2^{jq}\,\mb E_X \mb E_\eps \l| \frac{{m\choose j}^{1/2} }{{N\choose j}^{1/2}} \sum_{(i_1,\ldots,i_j)\in \m A_N^{(j)}} \eps_{i_1}\ldots\eps_{i_j}\h{j}_m(X_{i_1},\ldots, X_{i_j}) \r|^q
\\
\leq 2^{jq}(q-1)^{jq/2} \mb E \l| \frac{{m\choose j} }{{N\choose j}}\sum_{(i_1,\ldots,i_j)\in \m A_N^{(j)}} \l(\h{j}_m(X_{i_1},\ldots, X_{i_j})\r)^2 \r|^{q/2}.
}
Next, Hoeffding's representation of the U-statistic (Fact \ref{fact:3}) together with Jensen's inequality yields that 
\be{
\mb E \l| \frac{{m\choose j} }{{N\choose j}}\sum_{(i_1,\ldots,i_j)\in \m A_N^{(j)}} \l(\h{j}_m(X_{i_1},\ldots, X_{i_j})\r)^2 \r|^{q/2} 
\leq \mb E \l|\frac{{m\choose j}}{\lfloor N/j \rfloor} \sum_{i=1}^{\lfloor N/j \rfloor} W_i\r|^{q/2},
}
where $W_i:=\l(\h{j}_m(X_{(i-1)j+1},\ldots,X_{ij}) \r)^2$. 
We are going to estimate $\mb E \max_{j=1,\ldots,\lfloor N/j\rfloor} W_j^p$ in two different ways. First, recall that 
\be{
\h{j}_m(x_1,\ldots,x_j) := (\pi_j h_m)(x_1,\ldots,x_j) = (\delta_{x_1} - P_X)\times\ldots\times(\delta_{x_j}-P_X)\times P_X^{m-j} h_m.
}
Therefore, $(\pi_j h)(x_1,\ldots,x_j)$ is a linear combination of $2^j$ terms of the form $\prod_{i\in I}\delta_{x_i} \,P_X^{m - | I |} \,h_m$, for all choices of $I\subseteq [j]$. Consequently, $\l|(\pi_j h_m)(x_1,\ldots,x_j)\r|^2 \leq 2^{2j} \|h_m\|^2_\infty$, and the same bound also holds (almost surely) for the maximum of $W_j$'s. Therefore, $\mb E \max_{j=1,\ldots,\lfloor N/j\rfloor} W_j^p \leq 2^{2jp}\|h_m\|^{2p}_\infty$
and $\mb E \l( {m\choose j}W_1 \r)^p \leq (2e)^{2jp}\l( \frac{m}{j}\r)^{jp} \|h_m\|_\infty^{2p}$. 
Moreover, equation \eqref{eq:a02} in the proof of Theorem \ref{th:U-stat} implies that 
$\mb E W_1 \leq \frac{\var(h_m)}{{m\choose j}}$. Therefore, 
Rosenthal's inequality for nonnegative random variables (Fact \ref{fact:rosenthal}) entails that for $q\geq 2$,
\ml{
\mb E \l|\frac{{m\choose j}}{\lfloor N/j \rfloor} \sum_{i=1}^{\lfloor N/j \rfloor} W_i\r|^{q/2} 
\leq C^{q/2}\l( \var^{q/2}(h_m + \l( \frac{q}{2}\r)^{q/2} \l( \frac{j}{N}\r)^{q/2} \mb E \l( {m\choose j}\max_{j=1,\ldots,\lfloor N/j \rfloor}W_1 \r)^{q/2} \r)
\\
\leq C^{q/2}\l( \var^{q/2}(h_m) + \l( \frac{q}{2}\r)^{q/2} \l( \frac{j}{N}\r)^{q/2} (2e)^{jq}\l( \frac{m}{j}\r)^{jq/2} \|h_m\|_\infty^{q} \r)
}
and 
\be{
\mb E|V_{N,j}|^q \leq (Cq^{1/2})^{qj}\l( \var^{q/2}(h_m) \vee \l( \l( \frac{qj}{N}\r)^{1/2} \l( \frac{m}{j}\r)^{j/2} \|h_m\|_\infty \r)^q \r).
}
Markov's inequality therefore yields that 
\be{
\pr{|V_{N,j}|\geq (C_1 q)^{j/2}\l( \var^{1/2}(h_m)\vee  \l( \frac{qj}{N}\r)^{1/2} \l( \frac{m}{j}\r)^{j/2} \|h_m\|_\infty \r)} \leq e^{-q}.
}
Let $A(q) = (C_1 q)^{j/2}\var^{1/2}(h_m)$ and $B(q) = \|h_m\|_\infty\l( \frac{qj}{N}\r)^{1/2}\l( C_1 q^{1/2} \l( \frac{m}{j}\r)^{1/2} \r)^{j}$. If $t=A(q)\vee B(q)$, then $q = A^{-1}(t)\wedge B^{-1}(t)$. 
%For $q\geq \frac{\l|\log(\var(h_m))\r|}{j}$, 
We can solve the inequalities explicitly to get, after some algebra, that 
\ben{
\label{eq:bound-f1}
\pr{|V_{N.j}|\geq t} \leq \exp{ \min\l( \frac{1}{c}\l(\frac{t^2}{\var(h_m)}\r)^{{\frac{1}{j}}}, \frac{\l( \frac{t}{\|h_m\|_\infty} \sqrt{\frac N j}\r)^{\frac{2}{j+1}}}{ \l( \frac{cm}{j} \r)^{\frac{j}{j+1}}} \r)}.
}
\begin{remark}
Whenever $|X_1 - \mb EX_1|\leq M$ almost surely, the inequality $\l|(\pi_j h_m)(x_1,\ldots,x_j)\r| \leq 2^{j} \|h_m\|_\infty$ can be replaced by the bound $\l| (\pi_j h_m)(x_1,\ldots,x_j)\r| \leq C \l\|\partial_{u_j}\ldots\partial_{u_1} f_j \r\|_\infty (2M)^j$ which follows from Lemma \ref{lemma:variance} below. Combined with the assumption stating that $\l\| \partial_{u_j}\ldots\partial_{u_1} f_j \r\|_\infty \leq \l(\frac{C_1(P)}{m}\r)^{j/2} j^{\gamma_1 j}$, one easily finds that the resulting concentration inequality reads as follows:
\ben{
\label{eq:bound-f1.1}
\pr{|V_{N.j}|\geq t} \leq \exp{ \min\l( \frac{1}{c}\l(\frac{t^2}{\var(h_m)}\r)^{{\frac{1}{j}}}, 
\l(  \frac{t\sqrt{N/j}}{\l(c M j^{\gamma_1-1/2}\r)^j} \r)^{\frac{2}{j+1}} \r)}.
}
This bound holds for all $t>0$ and is usually sharper than \eqref{eq:bound-f1}.
\end{remark}

The bound \eqref{eq:bound-f1} is mostly useful only when $\frac{m}{j}$ is not too large. Now we will present a second way to estimate $\mb E \max_{j=1,\ldots,\lfloor N/j\rfloor} W_j^p$ that will yield much better inequalities for small values of $j$ and is valid when $X_1$ is not necessarily supported on a bounded interval. The key technical element that we rely on is the following lemma that allows one to control the growth of moments of $W_1$ with respect to $m$. Define 
\[
f_j(x_1,\ldots,x_j):=\mb E h_m(x_1,\ldots,x_j,X_{j+1},\ldots,X_m).
\]
\begin{lemma}
\label{lemma:variance}
Let conditions of the theorem hold and let $\sigma^2=\var(X_1)$. 
Then there exists $C=C(P)>0$ such that 
\[
\l| (\pi_j h_m)(X_1,\ldots,X_j)\r| \leq C  \l\|\partial_{u_j}\ldots\partial_{u_1} f_j \r\|_\infty \prod_{i=1}^j \l( |X_i-\mb EX_i| + \sigma\r)
\]
with probability $1$. Moreover, for any $p>2$, 
\be{
\mb E \l| (\pi_j h_m)(X_1,\ldots,X_j)\r|^p \leq 
C^{pj}\, \l\|\partial_{u_j}\ldots\partial_{u_1} f_j \r\|^p_\infty \l(\mb E \l|X_1 - \mb EX_1\r|^p\r)^j.
}
\end{lemma}
\noindent The proof of the lemma is outlined in section \ref{proof:variance}. As $\l\| \partial_{u_j}\ldots\partial_{u_1} f_j \r\|_\infty \leq \l(\frac{C_1(P)}{m}\r)^{j/2} j^{\gamma_1 j}$ by assumption, the second bound of the lemma can be written as 
\[
\mb E \l| (\pi_j h_m)(X_1,\ldots,X_j)\r|^p \leq 
C_2^{pj} m^{-jp/2} j^{\gamma_1 pj} \l(\mb E \l| X_1-\mb EX_1\r|^{p}\r)^j.
\]
Recall that $\nu_k = \mb E^{1/k} |X_1-\mb EX_1|^k$ and that under the stated assumptions, $\nu_k \leq k^{\gamma_2}M$ for all integers $k\geq 2$ and some $\gamma_2,M>0$. 
Therefore,
\ben{
\label{eq:moment01}
\mb E W_1^p \leq  C^{2pj} j^{2\gamma_1 p\,j} m^{-pj}\nu_{2p}^{2pj}
\leq \l( C' M j^{\gamma_1} m^{-1/2} p^{\gamma_2} \r)^{2pj},
}
and consequently $\mb E \l( {m\choose j}W_1 \r)^p \leq  \l( C' M j^{\gamma_1-1/2} p^{\gamma_2} \r)^{2pj}$. 
The rest of the argument proceeds in a similar way as before. Recall again that $\mb E W_1 \leq \frac{\var(h_m)}{{m\choose j}}$. Rosenthal's inequality for nonnegative random variables (Fact \ref{fact:rosenthal}) implies that for $q\geq 2$,
\be{
\mb E \l|\frac{{m\choose j}}{\lfloor N/j \rfloor} \sum_{i=1}^{\lfloor N/j \rfloor} W_i\r|^{q/2} 
\leq C^{q/2}\l( \var^{q/2}(h_m) + \l( \frac{q}{2}\r)^{q/2} \l( \frac{j}{N}\r)^{q/2} \mb E \l( {m\choose j}\max_{j=1,\ldots,\lfloor N/j \rfloor}W_1 \r)^{q/2} \r).
}
With the inequality for $\mb E W_1^p$ in hand, the expectation $\mb E \l( {m\choose j}\max_{j=1,\ldots,\lfloor N/j \rfloor}W_1 \r)^{q/2}$ can be upper bounded in two ways: first, trivially, 
\[
\mb E \l( {m\choose j}\max_{j=1,\ldots,\lfloor N/j \rfloor}W_1 \r)^{q/2}\leq \lfloor N/j \rfloor \mb E \l( {m\choose j} W_1 \r)^{q/2} \leq \lfloor N/j \rfloor  \l( C_1M j^{\gamma_1-1/2} q^{\gamma_2} \r)^{qj}. 
\]
On the other hand, for any identically distributed $\xi_1,\ldots,\xi_k$ and any $p>1$, $\mb E\max_{j=1,\ldots,k}|\xi_j| \leq k^{1/p} \max_{j=1,\ldots,k}\mb E^{1/p}|\xi_j|^p$. Choosing $\xi_j =  {m\choose j} W_j$ and $p = \lfloor \log(N/j) \rfloor +1$, we obtain the inequality 
\[
\mb E \l( {m\choose j}\max_{j=1,\ldots,\lfloor N/j \rfloor}W_1 \r)^{q/2}\leq 
\l( \log(N/j)\r)^{\gamma_2 qj} \l( C_1M j^{\gamma_1-1/2} q^{\gamma_2} \r)^{qj}.
\]
The second bound is better for $q\leq \frac{\log(N/j)}{\gamma_2 j\log\log(N/j)}$, therefore we get an estimate
\be{
\mb E \l|\frac{{m\choose j}}{\lfloor N/j \rfloor} \sum_{i=1}^{\lfloor N/j \rfloor} W_i\r|^{q/2} 
\leq C^{q/2}\l( \var^{q/2}(h_m) + \l( C_3^j\l( \frac{qj}{N}\r)^{1/2} 
\l(\log^{\gamma_2 }(N/j) Mj^{\gamma_1-1/2} q^{\gamma_2}\r)^j \r)^q \r)
}
and 
\be{
\mb E|V_{N,j}|^q \leq (Cq^{1/2})^{qj}\l( \var^{q/2}(h_m) \vee \l( \l( \frac{qj}{N}\r)^{1/2} 
\l(\log^{\gamma_2 }(N/j) Mj^{\gamma_1-1/2} q^{\gamma_2}\r)^j \r)^q \r)
}
that we will use for $2\leq q\leq \frac{\log(N/j)}{\gamma_2 j}$, while for larger values of $q$, $(N/j)^{1/q}\leq e^{\gamma_2 j}$ and 
\be{
%\mb E \l|\frac{{m\choose j}}{\lfloor N/j \rfloor} \sum_{i=1}^{\lfloor N/j \rfloor} W_i\r|^{q/2} 
\mb E|V_{N,j}|^q
\leq (Cq^{1/2})^{qj}\l( \var^{q/2}(h_m) \vee \l( \l( \frac{qj}{N}\r)^{1/2} \l(Mj^{\gamma_1-1/2} q^{\gamma_2}\r)^j \r)^q \r).
}
Markov's inequality therefore yields that for small values of $q$ (that is, whenever $2\leq q \leq \frac{\log(N/j)}{\gamma_2 j}$),
\be{
\pr{|V_{N,j}|\geq (C q)^{j/2}\l( \var^{1/2}(h_m)\vee\l( \frac{qj}{N}\r)^{1/2}\l( \log^{\gamma_2 }(N/j) M j^{\gamma_1-1/2} q^{\gamma_2} \r)^{j} \r)} \leq e^{-q}.
}
Let $A(q) = (C q)^{j/2}\var^{1/2}(h_m)$ and $B(q) = \l( \frac{qj}{N}\r)^{1/2}\l( Cq^{1/2}\log^{\gamma_2 }(N/j) M j^{\gamma_1-1/2} q^{\gamma_2} \r)^{j}$. If $t=A(q)\vee B(q)$, then $q = A^{-1}(t)\wedge B^{-1}(t)$. 
%For $q\geq \frac{\l|\log(\var(h_m))\r|}{j}$, we can 
Solving these inequalities explicitly to get, after some algebra, that 
\be{
\pr{|V_{N,j}|\geq t}\leq \exp{ \min\l( \frac{1}{c}\l(\frac{t^2}{\var(h_m)}\r)^{{\frac{1}{j}}}, \l( \frac{t\sqrt{N/j}}{\l(c \log^{\gamma_2}(N/j) M j^{\gamma_1-1/2}\r)^j}\r) \r)^{\frac{2}{1+j(2\gamma_2+1)}} }
}
for values of $t$ satisfying $2\leq \min\l( \frac{1}{c}\l(\frac{t^2}{\var(h_m)}\r)^{{\frac{1}{j}}}, \l( \frac{t\sqrt{N/j}}{\l(c \log^{\gamma_2}(N/j) M j^{\gamma_1-1/2}\r)^j}\r)^{\frac{2}{1+j(2\gamma_2+1)}} \r) \leq \frac{\log(N/j)}{\gamma_2 j}$. 
Similarly, for $q\geq \max\l(2, \frac{\log(N/j)}{\gamma_2 j}\r)$, the previously established bounds yield that 
\be{
\pr{|V_{N,j}|\geq (Cq)^{j/2}\l( \var^{1/q}(h_m) \vee \l( \frac{qj}{N}\r)^{1/2} \l(Mj^{\gamma_1-1/2} q^{\gamma_2}\r)^j  \r)} \leq e^{-q},
}
or equivalently 
\ben{
\label{eq:bound-f2}
\pr{|V_{N,j}|\geq t}
\leq \exp{ \min\l( \frac{1}{c}\l(\frac{t^2}{\var(h_m)}\r)^{{\frac{1}{j}}}, \l( \frac{t\sqrt{N/j}}{\l(c M j^{\gamma_1-1/2}\r)^j}\r) \r)^{\frac{2}{1+j(2\gamma_2+1)}} }
}
whenever $ \min\l( \frac{1}{c}\l(\frac{t^2}{\var(h_m)}\r)^{{\frac{1}{j}}}, \l( \frac{t\sqrt{N/j}}{\l(c M j^{\gamma_1-1/2}\r)^j}\r) \r)^{\frac{2}{1+j(2\gamma_2+1)}}  \geq  \max\l(2, \frac{\log(N/j)}{\gamma_2 j}\r)$. 
Combination of inequalities \eqref{eq:bound-f1} and \eqref{eq:bound-f2} yields the final result.

%###########################################################
\subsection{Proof of Lemma \ref{lemma:variance}.}
\label{proof:variance}
%###########################################################

Recall that $f_j(x_1,\ldots,x_j)=\mb E h_m\l(\frac{x_1}{\sqrt m},\ldots,\frac{x_j}{\sqrt m},\frac{X_{j+1}}{\sqrt m},\ldots,\frac{X_m}{\sqrt m}\r)$
%$h(x_1,\ldots,x_m) = \rho'\l( \frac{1}{\sqrt m}\sum_{i=1}^m x_i - t\sqrt{\frac m N} \r)$ and  
%\[f_j(x_1,\ldots,x_j):=\mb E h(x_1,\ldots,x_j,X_{j+1},\ldots,X_m)\] 
where $j<m$. It is easy to see from the definition of $\pi_j$ that $(\pi_j h)(x_1,\ldots,x_j) = (\pi_j f_j)(x_1,\ldots,x_j)$. 
Next, observe that for any function $g:\mb R^{j-1}\mapsto \mb R$ of $j-1$ variables such that $\mb Eg^2(X_1,\ldots,X_{j-1})<\infty,$ 
$\pi_j g = 0$ $P^{j-1}$-almost everywhere. Indeed, this follows immediately from the definition \eqref{eq:pi} of the operator $\pi_j$ since $g$ is a constant when viewed as a function of $y_j$.
%so that $(\delta_{y_j} - P) g = 0$. 
Based on this fact, it is easy to see that for any constant $a\in \mb R$, $f_j(x_1,\ldots,x_j)$ and $f_j(x_1,\ldots,x_j) - f_j|_{x_1=a}(x_2,\ldots,x_j)$, where $f_j|_{x_1=a}(x_2,\ldots,x_j) := f_j(a,x_2,\ldots,x_j)$, are mapped to the same function by $\pi_j$. In particular, $(\pi_j h)(x_1,\ldots,x_j) = \l(\pi_j (f_j - f_j|_{x_1=a})\r)(x_1,\ldots,x_j)$. 
Moreover, 
\be{
f_j(x_1,\ldots,x_j) - f_j|_{x_1=a}(x_2,\ldots,x_j) = \int_a^{x_1} \partial_{u_1} f_j(u_1,x_2,\ldots,x_j) du_1
%\mb E\rho''\l(\frac{\sum_{i=2}^{j} x_i}{\sqrt m} + \frac{u_1}{\sqrt m} + \frac{\sum_{i=j+1}^m X_i}{\sqrt m} - t \sqrt{\frac{m}{N}}\r) du_1.
}
Next, we repeat the same argument with $f_j$ replaced by 
\[
f_{j,2}(x_2,\ldots,x_j; u_1):=\partial_{u_1} f_j(u_1,x_2,\ldots,x_j) 
%\mb E\rho''\l(\frac{\sum_{i=2}^{j} x_i}{\sqrt m} + \frac{u_1}{\sqrt m} + \frac{\sum_{i=j+1}^m X_i}{\sqrt m} - t \sqrt{\frac{m}{N}}\r)
\] 
and noting that 
\be{
f_{j,2}(x_2,\ldots,x_j;u_1) - f_{2,j}|_{x_2 = a}(x_3,\ldots,x_j;u_1)
= \int_{a}^{x_2} \partial_{u_2} f_{j,2}(u_2,x_3,\ldots,x_j;u_1) du_2.
%\mb E\rho^{(3)}\l(\frac{\sum_{i=3}^{j} x_i}{\sqrt m} + \frac{u_1+u_2}{\sqrt m} + \frac{\sum_{i=j+1}^m X_i}{\sqrt m} - t \sqrt{\frac{m}{N}}\r) du_2.
}
The expression $\int_a^{x_1} f_{j,2} |_{x_2=a}(x_3,\ldots,x_j;u_1) du_1$ is a function of $j-1$ variables, hence $\pi_j$ maps it to $0$ so that 
\be{
(\pi_j h_m)(x_1,\ldots,x_j) 
= \pi_j \l( \int_a^{x_1}\int_a^{x_2} \partial_{u_2} f_{j,2}(u_2,x_3,\ldots,x_j;u_1) du_2 du_1 \r). 
%(x_1,\ldots,x_j).
%\mb E\rho^{(3)}\l(\frac{\sum_{i=3}^{j} x_i}{\sqrt m} + \frac{u_1+u_2}{\sqrt m} + \frac{\sum_{i=j+1}^m X_i}{\sqrt m} - t \sqrt{\frac{m}{N}}\r) du_2 du_1\r) (x_1,\ldots,x_j).
}
Iterating this process, we arrive at the expression 
\ben{
\label{eq:f03}
(\pi_j h_m)(x_1,\ldots,x_j) 
= \pi_j \l( \int_a^{x_1}\ldots\int_a^{x_j} \partial_{u_j}\ldots\partial_{u_1} f_j(u_1,\ldots,u_j) du_j\ldots du_1\r). 
%(x_1,\ldots,x_j).
%\mb E\rho^{(j+1)}\l(\frac{\sum_{i=1}^{j} u_i}{\sqrt m} + \frac{\sum_{i=j+1}^m X_i}{\sqrt m} - t \sqrt{\frac{m}{N}}\r) du_j\ldots du_1\r)(x_1,\ldots,x_j).
}	
\noindent Next, observe that for any function $g$ of $j$ variables,
\be{
(\pi_j g)(x_1,\ldots,x_j)=(\delta_{x_1} - P_X)\times\ldots\times(\delta_{x_j}-P_X) g 
= \mb E_{\tilde X} \l[ (\delta_{x_1} - \delta_{\tilde X_1})\times\ldots\times(\delta_{x_j} - \delta_{\tilde X_j})g\r], 
}
where $\tilde X_1,\ldots,\tilde X_j$ are i.i.d. with the same law as $X$, and independent from $X_1,\ldots,X_N$. Therefore, 
$(\pi_j h_m)(x_1,\ldots,x_j)$ is a linear combination of $2^j$ terms of the form $\mb E_{\tilde X}\l(\prod_{i\in I}\delta_{x_i}\prod_{j\in I^c}\delta_{\tilde X_j} \,g \r)$, for all choices of $I\subseteq [j]$ and
\[
g(x_1,\ldots,x_j) = \int_a^{x_1}\ldots\int_a^{x_j} \partial_{u_j}\ldots\partial_{u_1} f_j(u_1,\ldots,u_j) du_j\ldots du_1.
\] 
Take $a:=\mb EX_1$, and note that  
\ml{
\l|(\pi_j h_m)(x_1,\ldots,x_j)\r| \leq \l\| \partial_{u_j}\ldots\partial_{u_1} f_j \r\|_\infty\sum_{I\subseteq [j]}  \prod_{i\in I}|x_i-a| \prod_{j\in I^c} \mb E|\tilde X_i - a|
\\
\leq \l\| \partial_{u_j}\ldots\partial_{u_1} f_j \r\|_\infty \sum_{I\subseteq [j]}  \prod_{i\in I}|x_i-a| \cdot \sigma^{|I^c|} 
=  \l\| \partial_{u_j}\ldots\partial_{u_1} f_j \r\|_\infty \prod_{i=1}^j \l( |x_i-\mb EX_1| +\sigma\r).
}
The first claim of the lemma follows. To deduce the moment bound, observe that since $X_1,\ldots,X_j, \tilde X_1,\ldots,\tilde X_j$ are i.i.d. and in view of convexity of the function $x\mapsto |x|^p$ for $p\geq 1$,
\ml{
%\frac{2^p}{\l( m-j \r)^{pj/2}} \,
\mb E \l| (\pi_j h_m)(X_1,\ldots,X_j)\r|^p \leq 2^{(p-1)j} 
\mb E \l|\int_a^{X_1}\ldots\int_a^{X_j} \partial_{u_j}\ldots\partial_{u_1} f_j(u_1,\ldots,u_j)  \,du_j\ldots du_1\r|^p 
\\
\leq 2^{(p-1)j}  \l\| \partial_{u_j}\ldots\partial_{u_1} f_j \r\|^p_\infty \mb E \l| (X_1-\mb EX_1)\ldots (X_j-\mb EX_j)\r|^p.
}
for $a=\mb EX_1$.

%###########################################################
\subsection{Proof of Lemma \ref{lemma:truncation}.}
%###########################################################

As $\psi(x)$ is integrable, its Fourier transform equals $C_2 \wh\phi_1 \ast \wh\chi_R$, while $\wh\chi_R = \wh \kappa_{R}\cdot \wh I_{2R}$. It is well known \citep[e.g.][]{johnson2015saddle} that $\wh\kappa(x)\leq C_3 e^{-\sqrt{|x|}}$, hence $\wh\kappa_R(x) = \wh\kappa(Rx)\leq C_3 e^{-\sqrt{R|x|}}$. Moreover, $\wh I_{2R}(x) = \frac{\sin(2Rx)}{x}$. Therefore, for $|x|$ large enough, 
\ml{
\l| \wh\psi(x)\r| = C_2\l| \int_\mb R \wh \phi_1(x-y) \wh\chi_R(y)dy \r| 
\\
= C_2\l( \int\limits_{y: |y-x|\geq |x|/2} \wh \phi_1(x-y) \wh\chi_R(y)dy +  \int\limits_{y: |y-x| < |x|/2} \wh \phi_1(x-y) \wh\chi_R(y)dy \r).
}
To estimate the first integral, note that $\wh \phi_1(x-y) \leq \frac{C_1}{(1 + |x|/2)^\delta}\leq \frac{C_1 2^\delta}{(1+|x|)^\delta}$ whenever $|y-x|\geq |x|/2$ and that $\wh I_{2R}(x)\leq 2R$, implying that 
\[
\Bigg|\int\limits_{y: |y-x|\geq |x|/2} \wh \phi_1(x-y) \wh\chi_R(y)dy\Bigg| \leq \frac{C_4}{(1+|x|)^\delta} \int_\mb R  e^{-\sqrt{R|x|}} d(Rx) = \frac{C_5}{(1+|x|)^\delta}.
\]
On the other hand, 
\ml{
\Bigg|\int\limits_{y: |y-x| < |x|/2} \wh \phi_1(x-y) \wh\chi_R(y)dy\Bigg| \leq C_6 \l| \int_{x-|x|/2}^{x+|x|/2} e^{-\sqrt{R|x|}} \frac{\sin(2Rx)}{x} dx\r| 
\\
\leq C_7  \int_{R|x|/2}^{3R|x|/2} e^{-\sqrt{z}} dz\leq C_8 e^{-\sqrt{R|x|/2}}\sqrt{R|x|}.
}
Clearly, the last expression is smaller than $\frac{C_9}{(1+|x|)^\delta}$, implying the desired result.
%\leq C_4 \int_\mb R \frac{1}{(1+|x-y|)^\delta} 

%###########################################################
\subsection{Proof of Lemma \ref{lemma:deriv-bound}.}
\label{proof:deriv-bound}
%###########################################################

%\begin{lemma}
%Suppose that assumption XX holds. Then for $m$ large enough and $j=o(m/\log m)$,
%\be{
%\l\| \phi_{m-j}^{(j-1)} \r\|_\infty \leq \frac{C}{\sigma^{j}} \l(\frac{2j}{e}\r)^{j/2}
%}
%for some large enough constant $C=C(P)$.
%\end{lemma}
The proof proceeds using the standard Fourier-analytic tools. Let $\wh \phi_1:=\m F[\phi_1]$ be the Fourier transform of $\phi_1$, whence $\m F\l[\phi_{m-j}\r](t) = \l( \wh \phi_1\l(\frac{t}{\sqrt{m-j}}\r)\r)^{m-j}$. Therefore, 
\be{
\phi_{m-j}^{(j-1)}(t) = \frac{1}{2\pi}\int_\mb R \exp{-itx} (ix)^{j-1}\l( \wh \phi_1\l(\frac{x}{\sqrt{m-j}}\r)\r)^{m-j}dx
}
and 
$\l\| \phi_{m-j}^{(j-1)} \r\|_\infty \leq \frac{1}{2\pi} \int_\mb R |x|^{j-1} \l| \wh \phi_1\l(\frac{x}{\sqrt{m-j}}\r)\r|^{m-j}dx 
= \frac{(m-j)^{j/2}}{2\pi}\int_\mb R |x|^{j-1} \l| \wh\phi_1(x)\r|^{m-j} dx$. 
As $\l| \wh\phi_1(x)\r| \leq \frac{C_1}{(1+|x|)^\delta}$ by assumption, the integral is finite when $\delta(m-j)>j$ (in particular, this inequality holds when $m$ is large enough and $j=o(m)$ as $m\to\infty$). 
To get an explicit bound, we will estimate the integral over $[-\eta,\eta]$ and $\mb R \setminus [-\eta,\eta]$ separately, for a specific choice of $\eta>0$. To this end, observe that $\wh\phi_1(x) = \psi_\sigma(x)+o(x^2)$ where $\psi_\sigma(x)=\exp{-\frac{\sigma^2 x^2}{2}}$ is the characteristic function of the normal law $N(0,\sigma^2)$. Therefore, there exists $\eta>0$ such that for all $|x|\leq \eta$, $\l| \wh\phi_1(x)\r| \leq \exp{-\frac{\sigma^2 x^2}{4}}$, and 
\ml{
(m-j)^{j/2}\int_{-\eta}^\eta  |x|^{j-1} \l| \wh\phi_1(x)\r|^{m-j} dx
\leq 
(m-j)^{j/2}\int_{\mb R} |x|^{j-1} \exp{-\frac{\sigma^2 x^2 (m-j)}{4}} dx 
\\
= \int_\mb R |y|^{j-1} \exp{-\frac{\sigma^2 y^2}{4}}dy = \frac{2^{j}}{\sigma^{j}} \Gamma\l( \frac{j}{2}\r)
}     
where we used the exact expression for the absolute moments of the normal distribution. As 
$\Gamma(x+1)\leq C_2 \sqrt{2\pi x}\l(\frac{x}{e} \r)^{x}$ for all $x\geq 1$ and an absolute constant $C_2$ large enough, $ \frac{2^{j}}{\sigma^{j}} \Gamma\l( \frac{j}{2}\r) \leq \frac{C_2}{\sigma^{j}} \l(\frac{2j}{e}\r)^{j/2}$. At the same time,
\ml{
(m-j)^{j/2}\int_{\mb R\setminus [-\eta,\eta]}  |x|^{j-1} \l| \wh\phi_1(x)\r|^{m-j} dx 
= (m-j)^{j/2}\int_{\mb R\setminus [-(2C_1)^{2/\delta},(2C_1)^{2/\delta}]}  |x|^{j-1} \l| \wh\phi_1(x)\r|^{m-j} dx 
\\
+ (m-j)^{j/2}\int_{[-(2C_1)^{2/\delta},(2C_1)^{2/\delta}]\setminus [-\eta,\eta]}  |x|^{j-1} \l| \wh\phi_1(x)\r|^{m-j} dx 
}
where $C_1\geq 1$ is a constant such that $\l| \wh\phi_1(x)\r| \leq \frac{C_1}{(1+|x|)^\delta}$. The first term can be estimated via
\ml{
(m-j)^{j/2}\int_{\mb R\setminus [-(2C_1)^{2/\delta},(2C_1)^{2/\delta}]}  |x|^{j-1} \l| \wh\phi_1(x)\r|^{m-j} dx 
\\
\leq 
C_1^{m-j} (m-j)^{j/2} \int_{\mb R\setminus [-(2C_1)^{2/\delta},(2C_1)^{2/\delta}]}  \frac{|x|^{j-1}}{(1+|x|)^{\delta(m-j)}} dx
\\
\leq \frac{2 C_1^{m-j} (m-j)^{j/2}}{\delta (m - j) - j } \frac{1}{(2C_1)^{2(m - j) - 2j/\delta }}.
}
Whenever $m > 2j+2j/\delta $, we can bound the last expression from above by 
$C_3 m^{j/2} 2^{-m}$. Finally, as $\sup_{|x|>\eta}| \wh \phi_1(x) | \leq 1-\gamma$ for some $0<\gamma<1$,
\be{
(m-j)^{j/2}\int_{[-(2C_1)^{2/\delta},(2C_1)^{2/\delta}]\setminus [-\eta,\eta]}  |x|^{j-1} \l| \wh\phi_1(x)\r|^{m-j} dx 
\leq 2(m-j)^{j/2} (1-\gamma)^{m-j} \frac{(2C_1)^{2j/\delta}}{j}.
}
Putting the estimates together, we deduce that 
\ml{
\l\| \phi_{m-j}^{(j-1)} \r\|_\infty \leq \frac{(m-j)^{j/2}}{2\pi}\int_{\mb R}  |x|^{j-1} \l| \wh\phi_1(x)\r|^{m-j} dx 
\\
\leq \frac{C_2}{\sigma^{j}} \l(\frac{2j}{e}\r)^{j/2} + C_3 m^{j/2} 2^{-m} + C_4\l( (2C_1)^{4/\delta}m \r)^{j/2} (1-\gamma)^{m-j}.
}
Whenever $j = o(m/\log m)$, the last two terms in the sum above are negligible so that for $m$ large enough, 
\be{
\l\| \phi_{m-j}^{(j-1)} \r\|_\infty \leq  \frac{C_5}{\sigma^{j}} \l(\frac{2j}{e}\r)^{j/2},
}
as claimed.

\bibliography{MoM,robustERM} 
\bibliographystyle{apalike}

\end{document}